\documentclass[11pt,reqno]{amsart}

\usepackage[T1]{fontenc}
\usepackage{lmodern}
\usepackage{microtype}
\usepackage{mathtools}
\usepackage{amssymb}
\usepackage{mathrsfs}
\usepackage[margin=0.85in]{geometry}
\usepackage{enumitem}
\usepackage{xcolor}
\usepackage[colorlinks=true,linkcolor=blue!55!black,citecolor=blue!55!black,urlcolor=blue!55!black]{hyperref}
\hypersetup{
  pdftitle={Orlov's rationality conjecture for surfaces},
  pdfauthor={Xun Lin and Shizhuo Zhang},
  pdfsubject={Exceptional collections and rational surfaces over complex and arbitrary-characteristic fields},
  pdfkeywords={exceptional collection, rational surface, Kuznetsov height, blowup, positive characteristic}
}

\allowdisplaybreaks
\numberwithin{equation}{section}
\setlist[enumerate]{leftmargin=2.1em,itemsep=0.25em,topsep=0.4em}
\setlist[itemize]{leftmargin=2.0em,itemsep=0.25em,topsep=0.4em}

\newtheorem{theorem}{Theorem}[section]
\newtheorem{conjecture}[theorem]{Conjecture}
\newtheorem{proposition}[theorem]{Proposition}
\newtheorem{lemma}[theorem]{Lemma}
\newtheorem{corollary}[theorem]{Corollary}

\theoremstyle{definition}
\newtheorem{definition}[theorem]{Definition}
\theoremstyle{remark}

\newcommand{\Db}{\mathrm D^{\mathrm b}}
\newcommand{\Perf}{\operatorname{Perf}}
\newcommand{\RHom}{\operatorname{RHom}}

\newcommand{\Hom}{\operatorname{Hom}}
\newcommand{\Ext}{\operatorname{Ext}}
\newcommand{\Pic}{\operatorname{Pic}}
\newcommand{\Supp}{\operatorname{Supp}}
\newcommand{\Cone}{\operatorname{Cone}}
\newcommand{\id}{\operatorname{id}}
\newcommand{\NHH}{\operatorname{NHH}}
\newcommand{\HH}{\operatorname{HH}}
\newcommand{\tr}{\operatorname{tr}}
\newcommand{\Conf}{\operatorname{Conf}}
\newcommand{\cO}{\mathcal O}
\newcommand{\cL}{\mathcal L}
\newcommand{\cM}{\mathcal M}
\newcommand{\cN}{\mathcal N}
\newcommand{\cF}{\mathcal F}
\newcommand{\cE}{\mathcal E}
\newcommand{\cA}{\mathcal A}

\newcommand{\cI}{\mathcal I}
\newcommand{\cJ}{\mathcal J}

\newcommand{\bL}{\mathbf L}
\newcommand{\bR}{\mathbf R}
\newcommand{\resultcite}[2]{\textup{[#1, \hyperlink{cite.#2}{#2}]}}

\title[Orlov's rationality conjecture for surfaces]
{Orlov's rationality conjecture for surfaces}

\author{Xun Lin}
\address{School of Science and Engineering, The Chinese University of Hong Kong, Shenzhen, Shenzhen 518172, P.~R. China}
\email{linxun@cuhk.edu.cn}

\author{Shizhuo Zhang}
\address{School of Mathematics, Sun Yat-sen University, Guangzhou 510275, P.~R. China}
\email{zhangshzh28@mail.sysu.edu.cn}


\begin{document}

\begin{abstract}
We prove that a smooth projective surface over $\mathbb{C}$ admitting a full exceptional collection of line bundles is rational.
\end{abstract}
\maketitle

\section{Introduction}\label{sec:introduction}

We consider the following folklore conjecture attributed to Orlov.

\begin{conjecture}[{\resultcite{Introduction, p.~5}{Via17}}]\label{conj:orlov}
Let $X$ be a smooth connected projective surface over $\mathbb C$. If its bounded derived category of coherent sheaves $\Db(X)$ admits a full exceptional collection, then $X$ is rational.
\end{conjecture}

By the work of Hille and Perling, every smooth projective rational surface admits a full exceptional collection of line bundles \resultcite{Theorem~5.6}{HP11}. It is therefore natural to ask whether a smooth projective surface admitting a full exceptional collection of line bundles must be rational. This is precisely the line-bundle case of Conjecture~\ref{conj:orlov}. Our main result answers this question affirmatively.

\begin{theorem}\label{thm:main}
Let $X$ be a smooth connected projective surface over $\mathbb C$. If $\Db(X)$ admits a full exceptional collection of line bundles, then $X$ is rational.
\end{theorem}

Combining Theorem~\ref{thm:main} with \resultcite{Theorem~5.6}{HP11}, we obtain the following characterization of rational surfaces.

\begin{corollary}\label{cor:characterization}
A smooth connected projective complex surface is rational if and only if its bounded derived category admits a full exceptional collection of line bundles.
\end{corollary}

Rationality had previously been established for full strong exceptional collections of line bundles \resultcite{Theorem~4.4}{BS17}, and for strong exceptional collections of line bundles of maximal numerical length \resultcite{Theorem~1.2}{Zha19}. Under the additional cyclic strongness condition, a more precise conclusion was obtained: the surface is weak del Pezzo of degree at least three \resultcite{Theorem~1 and Proposition~6.2}{EXZ21}. Without strongness, rationality was proved for full exceptional collections of line bundles on surfaces of Picard number at most three \resultcite{Theorem~1.4}{Zha19}.

Theorem~\ref{thm:main} removes these restrictions. A first difficulty is that the numerical conditions on an exceptional collection also occur on nonrational surfaces. Numerically exceptional collections of line bundles of maximal length are characterized by a condition on the intersection lattice \resultcite{Theorem~3.1}{Via17}, and such collections exist on surfaces of general type with $p_g=q=0$ \resultcite{Theorem~3.10}{Via17}. To obtain rationality, we therefore need a geometric consequence of fullness.

A second difficulty is the passage to a minimal model. Each contraction must be accompanied by a construction of a full exceptional collection of line bundles on the contracted surface. After deforming the blowup centers, we construct this new collection from a pair of line bundles differing by an exceptional curve (Propositions~\ref{prop:3-28} and~\ref{prop:3-31}). This gives a contraction procedure that can be repeated until the minimal model is reached.

The result is valid in arbitrary characteristic as well. Appendix~\ref{sec:arbitrary-characteristic} proves Theorem~\ref{thm:main} over every algebraically closed field (Theorem~\ref{thm:arbitrary-characteristic}). Over an arbitrary ground field, a full exceptional collection of line bundles implies geometric rationality; if the field is perfect, it implies rationality over that field (Corollary~\ref{cor:arbitrary-ground-field}).

\subsection{Outline of the proof}\label{subsec:intro-proof}

Suppose that $X$ admits a full exceptional collection of line bundles and is not rational. Its minimal model $S$ satisfies $p_g(S)=q(S)=0$, and $K_S$ is nef. We prove that a surface with nef canonical class cannot admit such a collection, and then show how to descend the collection from a suitable blowup of $S$ to $S$ itself.

The obstruction on $S$ comes from Kuznetsov's height: positive height excludes fullness \resultcite{Proposition~6.1}{Kuz15}. We establish an estimate that also applies during the contraction process. Let $\cE=(\cL_1,\ldots,\cL_\gamma)$ be a nonempty exceptional collection of line bundles on a smooth projective surface $Y$. Proposition~\ref{prop:3-16} gives $h_Y(\cE)\ge1$ if $K_Y\cdot C\ge0$ for every prime component $C$ of the zero divisor of every nonzero map $\cL_i\to\cL_j$ with $i<j$. In particular, this condition holds when $K_Y$ is nef.

The estimate begins with the geometry of these zero divisors. Under the intersection condition above, multiplication by a defining section is injective on the relevant first cohomology groups (Proposition~\ref{prop:3-4}). Each zero divisor is also unique in its linear equivalence class, and divisors of maps with a common source or target are nested (Lemma~\ref{lem:3-2}). These properties determine the algebra of compositions. Writing $V=\bigoplus_i\cL_i$ and $R=\Hom_Y(V,V)$, we obtain a product of path algebras of linearly oriented chains; the injectivity just established makes $\Ext_Y^1(V,V)$ projective as a left $R$-module (Proposition~\ref{prop:3-8}).

In the normal Hochschild spectral sequence, the first differential is given by composition. We filter its graded bar complex by the number of factors of positive degree. Each resulting piece is computed by derived tensor products over $R$, followed by a derived trace (Lemma~\ref{lem:3-10}). The projectivity of $\Ext_Y^1(V,V)$ and the resolutions for a path algebra give vanishing in total degrees below one on the second page (Lemma~\ref{lem:3-15}). This proves the height estimate.

To apply it to a blowup of $S$, we move the centers away from the rational curves on $S$. The family of ordered point blowups includes towers with infinitely near centers, and the given line bundles extend over this family. Fullness is open, while $S$ contains at most countably many integral rational curves. We can therefore choose a member
\[
Y=\operatorname{Bl}_{p_1,\ldots,p_r}S
\]
carrying a full exceptional collection, with distinct centers lying on none of these curves (Corollary~\ref{cor:3-23}).

On $Y$, the zero divisor of a nonzero map between distinct members of the collection is either a pullback from $S$ or a single exceptional curve (Lemma~\ref{lem:3-24}). If all these divisors were pullbacks, the height estimate would contradict fullness. Thus two members have the form $\cL,\cL(E)$ for an exceptional curve $E$. Proposition~\ref{prop:3-28} shows that such a pair can be chosen adjacent. Mutations then replace the pair by a sheaf on $E$ and a line bundle; after moving that sheaf to the beginning of the collection, a common twist makes the remaining line bundles descend to the contraction of $E$ (Proposition~\ref{prop:3-31}). The remaining centers still avoid the rational curves of $S$, so the procedure can be repeated. It produces a full exceptional collection of line bundles on $S$, contradicting the height estimate.

Appendix~\ref{sec:arbitrary-characteristic} extends this argument to arbitrary characteristic. A Picard-scheme calculation replaces the finite-logarithm argument for nonreduced divisors (Lemma~\ref{lem:appendix-picard}). The centers are chosen at a geometric generic point of the configuration space, where the incidence condition of Lemma~\ref{lem:appendix-generic-avoidance} holds on every partial blowup. The same height estimate and contraction procedure then apply.

\subsection{Organization of the article}\label{subsec:intro-organization}

Section~\ref{sec:preliminaries} fixes the conventions used throughout the proof. Subsections~\ref{subsec:exceptional-mutations} and~\ref{subsec:left-orthogonal} recall exceptional collections, mutations, and left-orthogonal divisors. Subsections~\ref{subsec:height} and~\ref{subsec:normal-spectral-sequence} introduce Kuznetsov's height and the bar spectral sequence that computes normal Hochschild cohomology.

Section~\ref{sec:rationality} carries out the argument over $\mathbb C$. Subsections~\ref{subsec:divisor-geometry}--\ref{subsec:bar-height} pass from the geometry of effective divisors to the algebra of compositions and then to the height bound. Subsection~\ref{subsec:general-centers} moves the blowup centers while preserving fullness. Subsections~\ref{subsec:exceptional-pair} and~\ref{subsec:contraction} find the adjacent exceptional-curve pair and descend the collection through its contraction. The induction is completed in Subsection~\ref{subsec:main-proof}, proving Theorem~\ref{thm:main} and Corollary~\ref{cor:characterization}.

Appendix~\ref{sec:arbitrary-characteristic} makes the changes needed in arbitrary characteristic. Subsection~\ref{subsec:appendix-literature} reviews the rationality question over general fields and its relation to the results of this appendix. Subsections~\ref{subsec:appendix-divisors}--\ref{subsec:appendix-height} replace the nonreduced Picard-group argument and carry over the height bound. Subsections~\ref{subsec:appendix-generic-centers}--\ref{subsec:appendix-contraction} establish the incidence condition and apply it to deformation and contraction. Subsection~\ref{subsec:appendix-proof} completes the proof over an algebraically closed field, and Subsection~\ref{subsec:appendix-ground-fields} proves the ground-field assertions.

\subsection{Acknowledgements}\label{subsec:acknowledgements}

Shizhuo Zhang is supported by the Sun Yat-sen University Starting Grant 34000-12256019.

\subsection{AI disclosure}\label{subsec:ai-disclosure}

This work was inspired by the rationality results in \resultcite{Theorems~1.2 and~1.4}{Zha19}. The proof uses a sharper estimate for Kuznetsov's height of exceptional collections \resultcite{Definition~4.1}{Kuz15}, established in Proposition~\ref{prop:3-16}. The approach was also inspired by the first author's refinement of height theory through the Hochschild--Serre mechanism. In future work, we will use this refined theory and its estimates to study full exceptional collections whose members need not all be line bundles.

The authors obtained a complete proof through several rounds of interaction with AI. AI also assisted with auditing the arguments for correctness, polishing the language, and drafting the manuscript. The authors take responsibility for the correctness of the proofs.

\section{Exceptional line bundles and Kuznetsov's height}\label{sec:preliminaries}

We recall the categorical constructions and the divisor identities used in the surface proof. In the main text we work over $\mathbb C$; the field conventions for the appendix are stated there. Varieties are smooth, connected, and projective unless otherwise stated. We use the usual dg enhancement of $\Db(X)$, and regard a vector bundle as a complex concentrated in degree zero. Our shift convention is $H^q(C[s])=H^{q+s}(C)$. For a Cartier divisor $D$, write $\cF(D)=\cF\otimes\cO_X(D)$. We write $\Pic(X)$ for the Picard group. For a line bundle $\cL$ and a curve $C$, we write $\cL\cdot C=c_1(\cL)\cdot C$.

\subsection{Exceptional collections and mutations}\label{subsec:exceptional-mutations}

An object $E\in\Db(X)$ is exceptional if $\RHom(E,E)\simeq\mathbb C$. A sequence $(E_1,\ldots,E_\gamma)$ of exceptional objects is an exceptional collection if $\RHom(E_j,E_i)=0$ whenever $j>i$. It is full if its members generate $\Db(X)$ as a thick triangulated subcategory. For a collection $\cE$, let $\langle\cE\rangle$ denote its thick triangulated span. This subcategory is admissible \resultcite{Lemma~2.10}{Hu19}: the functor $V\mapsto V\otimes E$ associated with one exceptional object has right adjoint $\RHom(E,-)$ and left adjoint $\RHom(-,E)^\vee$, and one iterates this construction along the collection. Thus $\cE$ gives a semiorthogonal decomposition
\[
\Db(X)=\langle\cA,E_1,\ldots,E_\gamma\rangle,\qquad \cA=\langle\cE\rangle^\perp.
\]
Here $\langle\cE\rangle^\perp$ consists of objects $F$ with $\RHom(E_i,F)=0$ for every $i$; fullness is equivalent to $\cA=0$.

For an exceptional object $E$, the left and right mutations through $E$ are denoted by $\bL_E$ and $\bR_E$, respectively. They are defined by the evaluation and coevaluation triangles
\begin{align*}
\RHom(E,F)\otimes E&\longrightarrow F\longrightarrow\bL_E(F),\\
\bR_E(F)&\longrightarrow F\longrightarrow\RHom(F,E)^\vee\otimes E.
\end{align*}
Relative to \resultcite{Definition~2.8}{Hu19}, our left and right mutations are shifted by $[1]$ and $[-1]$, respectively. For an exceptional pair $(E,F)$, the mutated pairs $(\bL_E(F),E)$ and $(F,\bR_F(E))$ have the same thick span by \resultcite{Lemma~2.11}{Hu19}. Tensoring every member of a collection by the same line bundle preserves exceptionality, strongness, and fullness.

The Euler pairing is
\[
\chi(E,F)=\sum_q(-1)^q\dim_{\mathbb C}\Ext^q(E,F).
\]
A collection is numerically exceptional if its Euler matrix is upper triangular with diagonal entries one. It has maximal numerical length when its length equals the rank of the numerical Grothendieck group $K_0^{\mathrm{num}}(X)$, the quotient of $K_0(X)$ by the radical of this pairing. Fullness implies maximal numerical length.

By Serre duality, the Serre functor on an $n$-fold and its inverse are
\resultcite{equation~(10)}{Kuz15}
\begin{equation}\label{eq:2-1}
S_X(F)=F\otimes\omega_X[n],\qquad S_X^{-1}(F)=F\otimes\omega_X^{-1}[-n].
\end{equation}

\subsection{Exceptional line bundles and left-orthogonal divisors}\label{subsec:left-orthogonal}

For line bundles $\cL_i=\cO_X(D_i)$,
\begin{equation}\label{eq:2-2}
\Ext^q(\cL_i,\cL_j)=H^q(X,\cO_X(D_j-D_i)).
\end{equation}
In particular, a line bundle is exceptional precisely when $H^q(X,\cO_X)=0$ for all $q>0$. On a surface this is equivalent to $p_g(X)=q(X)=0$. Under this assumption a divisor class $D$ is called left-orthogonal if $H^q(X,\cO_X(-D))=0$ for every $q$, and strongly left-orthogonal if in addition $H^q(X,\cO_X(D))=0$ for $q>0$. Replacing the first vanishing by $\chi(\cO_X(-D))=0$ defines numerical left-orthogonality. These notions are defined in \resultcite{Definition~3.1}{HP11}.

The numerical identities used below are the Riemann--Roch calculation in \resultcite{Lemma~3.3(i)--(ii)}{HP11}, which only requires $\chi(\cO_X)=1$. Explicitly,
\[
\chi(\cO_X(D))=1+\tfrac12(D^2-K_XD).
\]
Consequently every numerically left-orthogonal class satisfies
\begin{equation}\label{eq:2-3}
D^2+K_XD=-2,\qquad \chi(\cO_X(D))=-K_XD=D^2+2.
\end{equation}
By \eqref{eq:2-2}, $(\cL_1,\ldots,\cL_\gamma)$ is exceptional precisely when every forward difference $D_j-D_i$, $i<j$, is left-orthogonal. It is strong precisely when these differences are strongly left-orthogonal. A nonzero forward morphism is a section of $\cO_X(D_j-D_i)$ and has a nonzero effective zero divisor in this class. The divisor is nonzero because distinct members of an exceptional collection cannot be isomorphic. Effective divisors and inequalities between them are always taken with their multiplicities.

\subsection{Kuznetsov's height}\label{subsec:height}

We recall the normal Hochschild construction of \resultcite{Definition~3.2 and \S3.3}{Kuz15} in notation adapted to an exceptional collection. Choose representatives of $\cE=(E_1,\ldots,E_\gamma)$ in an enhancement of $\Db(X)$. Its directed dg category retains the actual morphism complexes from $E_i$ to $E_j$ for $i<j$, has $\mathbb C\id_{E_i}$ on the diagonal, and is zero in reverse entries. This is quasi-equivalent to the full dg subcategory on these objects. Let $\mathcal A$ be its algebra and put
\[
S_0=\bigoplus_i\mathbb C e_i,\qquad
\mathcal J=\bigoplus_{i<j}e_j\mathcal A e_i,\qquad
\mathcal A=S_0\oplus\mathcal J,
\]
where $e_j\mathcal A e_i=\RHom(E_i,E_j)$. The normal bimodule $\mathcal T$ has entries
\begin{equation}\label{eq:2-4}
e_i\mathcal T e_j=\RHom(E_j,S_X^{-1}E_i).
\end{equation}
Its two actions are defined in the enhancement by composition and composition after applying $S_X^{-1}$, respectively.

\begin{definition}[{\resultcite{Definitions~3.2 and~4.1}{Kuz15}}]\label{def:2-1}
The normal Hochschild complex and height of $\cE$ are
\[
N_X(\cE)=\mathcal T\otimes^{\mathbf L}_{\mathcal A^{\mathrm{op}}\otimes\mathcal A}\mathcal A,
\qquad h_X(\cE)=\inf\{q:H^qN_X(\cE)\ne0\}.
\]
We set $\inf\varnothing=+\infty$ and write $\NHH^q(\cE,X)=H^qN_X(\cE)$.
\end{definition}

This is Kuznetsov's definition, with the description of the normal bimodule in \eqref{eq:2-4}; see \resultcite{Definition~3.2, equation~(19), Lemma~3.5, and Definition~4.1}{Kuz15}. In particular, the invariant depends on the enhanced embedding, including its composition maps.

\begin{proposition}[{\resultcite{Theorem~3.3 and Proposition~6.1}{Kuz15}}]\label{prop:2-2}
A full exceptional collection on a smooth connected projective variety has height zero. In particular, positive height obstructs fullness.
\end{proposition}
\begin{proof}
For a semiorthogonal decomposition $\Db(X)=\langle\cA,\cE\rangle$, the normal localization triangle is
\begin{equation}\label{eq:2-5}
N_X(\cE)\longrightarrow\HH(X)\longrightarrow\HH(\cA),
\end{equation}
by \resultcite{Theorem~3.3}{Kuz15}. If $\cE$ is full, then $\cA=0$ and the first arrow is a quasi-isomorphism. The Hochschild--Kostant--Rosenberg decomposition recalled in \resultcite{\S3.1}{Kuz15} gives $\HH^{<0}(X)=0$ and $\HH^0(X)=H^0(X,\cO_X)=\mathbb C$.
\end{proof}

\subsection{The normal Hochschild spectral sequence}\label{subsec:normal-spectral-sequence}

The following bar model and spectral sequence are those of \resultcite{Lemma~3.6, Proposition~3.7, and equation~(20)}{Kuz15}, written with our grading convention. Since $S_0$ is semisimple, the normalized relative bar resolution computes the derived tensor product in Definition~\ref{def:2-1}. Directionality gives $\mathcal J^{\otimes_{S_0}p}=0$ for $p\ge \gamma$. The normal complex therefore has a finite bar model whose length-$p$ term is
\begin{equation}\label{eq:2-6}
\bigoplus_{i_0<\cdots<i_p}
\RHom(E_{i_0},E_{i_1})\otimes\cdots\otimes
\RHom(E_{i_{p-1}},E_{i_p})\otimes
\RHom(E_{i_p},S_X^{-1}E_{i_0})[p].
\end{equation}
The differential is the sum of the internal differentials, signed adjacent compositions, and the two endpoint actions. For the action crossing the endpoint, the first arrow is transported by $S_X^{-1}$. The signs are the usual Koszul signs. This description is established in \resultcite{Proposition~3.7}{Kuz15}.

The filtration by bar length yields a convergent spectral sequence
\begin{equation}\label{eq:2-7}
\begin{aligned}
E_1^{-p,q}
&=\bigoplus_{\substack{i_0<\cdots<i_p\\q_0+\cdots+q_p=q}}
\Ext^{q_0}(E_{i_0},E_{i_1})\otimes\cdots\otimes
\Ext^{q_{p-1}}(E_{i_{p-1}},E_{i_p})\\
&\hspace{7em}\otimes\Ext^{q_p}(E_{i_p},S_X^{-1}E_{i_0})
\quad\Longrightarrow\quad\NHH^{q-p}(\cE,X).
\end{aligned}
\end{equation}
The first differential is Yoneda composition, including both endpoint actions. Thus it is the graded Hochschild chain differential for $A=H^*(\mathcal A)$ with coefficient bimodule $T=H^*(\mathcal T)$. The face maps and signs are given in \eqref{eq:3-10}--\eqref{eq:3-11}.

\section{Rationality of surfaces}\label{sec:rationality}

The proof has two parts. We first establish the height bound under a canonical-degree condition on effective differences. We then deform the blowup centers and descend a full exceptional collection through successive contractions. In Subsections~\ref{subsec:divisor-geometry}--\ref{subsec:bar-height}, $X$ is a smooth projective surface with $p_g(X)=q(X)=0$. Every nonzero forward morphism between exceptional line bundles defines an effective left-orthogonal divisor; its scheme structure, including component multiplicities, is retained throughout the argument.

\subsection{Geometry of effective left-orthogonal divisors}\label{subsec:divisor-geometry}

\begin{lemma}[\resultcite{Lemma~4.1 and Proposition~4.3}{Ela16}]\label{lem:3-1}
Let $D>0$ be an effective left-orthogonal divisor. Then
$H^0(D,\cO_D)=\mathbb C$ and $H^1(D,\cO_D)=0$. For every nonzero effective subdivisor $B\le D$,
\begin{equation}\label{eq:3-1}
 H^1(B,\cO_B)=0,
 \qquad B^2+K_XB=-2h^0(B,\cO_B)\le -2.
\end{equation}
Every prime component $C$ of $D$ is a smooth rational curve, with
$C^2=-2-K_XC$, and the support of $D$ is connected.
\end{lemma}

\begin{proof}
The following proof expands the restriction-sequence arguments in
\resultcite{Lemma~4.1 and the proof of Theorem~4.5}{Ela16} and the component argument in
\resultcite{Proposition~4.3}{Ela16}.
The sequence $0\to\cO_X(-D)\to\cO_X\to\cO_D\to0$ gives the first assertions.
The kernel of the surjection $\cO_D\to\cO_B$ is supported in dimension at most one.
Its second cohomology vanishes, so $H^1(\cO_D)\to H^1(\cO_B)$ is surjective.
Subtracting Riemann--Roch for $\cO_X(-B)$ from that for $\cO_X$ gives
\[
 \chi(\cO_B)=-\frac12(B^2+K_XB)=h^0(\cO_B)\ge1,
\]
which proves \eqref{eq:3-1}. Let $C_i\le D$ be a reduced prime component, and let
$\nu_i\colon\widetilde C_i\to C_i$ be its normalization. Since $C_i$ is integral,
$h^0(\cO_{C_i})=1$; by \eqref{eq:3-1}, $h^1(\cO_{C_i})=0$. The normalization sequence
\[
 0\longrightarrow\cO_{C_i}\longrightarrow(\nu_i)_*\cO_{\widetilde C_i}
 \longrightarrow Q_i\longrightarrow0,
 \qquad\dim\Supp Q_i\le0,
\]
gives
\[
 0=p_a(C_i)=g(\widetilde C_i)+\operatorname{length}Q_i,
 \qquad g(\widetilde C_i)\ge0,
 \qquad\operatorname{length}Q_i\ge0.
\]
Thus $Q_i=0$, $\nu_i$ is an isomorphism, and
$C_i\simeq\widetilde C_i\simeq\mathbb P^1$. Adjunction gives
$C_i^2+K_XC_i=2p_a(C_i)-2=-2$. Finally, disconnected support would give a
nonconstant idempotent in $H^0(\cO_D)=\mathbb C$.

\end{proof}

\begin{lemma}\label{lem:3-2}
Let $D>0$ be effective and left-orthogonal, with $K_XC\ge0$ for every prime component $C$ of $D$.
Then $h^0(X,\cO_X(D))=1$. Moreover, if $C_0,E_0>0$ are effective left-orthogonal divisors,
$K_X$ is nonnegative on every component of $E_0$, and $C_0-E_0$ is left-orthogonal,
then $E_0<C_0$ as actual divisors.
\end{lemma}

\begin{proof}
If $|D|$ moved, remove its fixed part and take a nonzero moving member $M\le D$.
Two general members of $|M|$ have no common component, giving $M^2\ge0$.
Since $K_XM\ge0$, this contradicts \eqref{eq:3-1}.
Thus $D$ is the unique effective divisor in its linear equivalence class.

For the second assertion, let $G$ be the componentwise common divisor of $C_0$ and $E_0$,
and write $B=C_0-G$, $F=E_0-G$. The divisors $B$ and $F$ have no common component.
If both are nonzero, then
\[
 \begin{aligned}
 (C_0-E_0)^2+K_X(C_0-E_0)
 &=(B^2+K_XB)+(F^2+K_XF)-2K_XF-2BF\\
 &\le-4,
 \end{aligned}
\]
contrary to \eqref{eq:2-3}. If $B=0$, the nonzero section of
$\cO_X(F)=\cO_X(-(C_0-E_0))$ contradicts left-orthogonality; this also excludes
$B=F=0$. Hence $F=0$ and $B>0$.
\end{proof}

\begin{lemma}\label{lem:3-3}
For an effective left-orthogonal divisor $D>0$, component degrees identify
\[
 \Pic(D)\simeq\mathbb Z^{\{\text{prime components of }D\}}.
\]
A line bundle $\cN$ of nonnegative degree on every prime component has a section that is
a nonzerodivisor on $D$. Consequently, for every line bundle $\cM$ on $D$,
\begin{equation}\label{eq:3-2}
 H^1(D,\cM)=0\quad\Longrightarrow\quad H^1(D,\cM\otimes\cN)=0.
\end{equation}
\end{lemma}

\begin{proof}
The description of the reduced support is \resultcite{Proposition~4.3}{Ela16}.
We expand that argument using normalization, before computing the Picard group of the possibly nonreduced divisor.
Write $D=\sum_{i=1}^a m_iC_i$ and $Y=D_{\mathrm{red}}=\bigcup_{i=1}^a C_i$.
By Lemma~\ref{lem:3-1},
\[
 H^0(Y,\cO_Y)=\mathbb C,
 \qquad H^1(Y,\cO_Y)=0,
 \qquad C_i\simeq\mathbb P^1.
\]
For the normalization $\nu\colon\widetilde Y=\coprod_i C_i\to Y$,
\[
 0\longrightarrow\cO_Y\longrightarrow\nu_*\cO_{\widetilde Y}
 \longrightarrow Q\longrightarrow0,
 \qquad\delta_p(Y):=\operatorname{length}Q_p.
\]
For $i\ne j$ and $p\in C_i\cap C_j$, choose local equations $f_i,f_j\in\cO_{X,p}$
for $C_i,C_j$. Their local intersection multiplicity is
\[
 I_p(C_i,C_j):=\operatorname{length}_{\cO_{X,p}}
 \bigl(\cO_{X,p}/(f_i,f_j)\bigr)
 =\dim_{\mathbb C}\bigl(\cO_{X,p}/(f_i,f_j)\bigr)\ge1.
\]
This length is finite since $C_i$ and $C_j$ have no common component.
Apply the two-divisor exact sequence of \resultcite{equation~(4.4)}{Ela16} successively to unions of branches.
On adjoining a smooth branch, comparison with normalization increases the cokernel length by its
intersection multiplicity with the existing union. This multiplicity is the sum of the pairwise
multiplicities, since the local ring of the new branch is a discrete valuation ring. It follows that
\[
 \delta_p(Y)=\sum_{\substack{i<j\\p\in C_i\cap C_j}}I_p(C_i,C_j),
 \qquad\sum_p\delta_p(Y)=\chi(\cO_{\widetilde Y})-\chi(\cO_Y)=a-1.
\]
Let $\Gamma$ be the simple graph with vertices $C_i$ and an edge $ij$ whenever
$C_i\cap C_j\ne\varnothing$. Connectedness of $Y$ gives
\[
 a-1\le |E(\Gamma)|\le\sum_{i<j}C_i\cdot C_j
 =\sum_p\delta_p(Y)=a-1.
\]
Thus $\Gamma$ is a tree and every nonzero $C_i\cdot C_j$ equals one.
Three components through one point would give a triangle in $\Gamma$.
Hence every singular point has exactly two transverse branches:
\[
 \widehat{\cO}_{Y,p}\simeq\mathbb C[[u,v]]/(uv)
 \qquad(p\in\operatorname{Sing}Y).
\]
Orient $\Gamma$ away from a root. For an edge $e\colon i\to j$, let
$p_e=C_i\cap C_j$ and $\iota_e\colon\{p_e\}\hookrightarrow Y$. Put
\[
 \mathcal U:=\nu_*\cO_{\widetilde Y}^{\times},
 \qquad\mathcal S:=\prod_{e\in E(\Gamma)}(\iota_e)_*\mathbb C^{\times}.
\]
Here $\cO_Y^{\times}$ denotes the sheaf of invertible regular functions; each factor of
$\mathcal S$ is the skyscraper sheaf with stalk $\mathbb C^{\times}$ at $p_e$.
The normalization sequence for units is
\[
 1\longrightarrow\cO_Y^{\times}\xrightarrow{\nu^*}\mathcal U
 \xrightarrow{q}\mathcal S\longrightarrow1,
 \qquad q((u_i))_e=\frac{u_j(p_e)}{u_i(p_e)}.
\]
Exactness can be checked at the nodes; away from them $\nu$ is an isomorphism.
At $p=p_e$, transversality gives
\[
 \begin{aligned}
 \cO_{Y,p}&\simeq\cO_{C_i,p}\times_{\mathbb C}\cO_{C_j,p},\\
 \cO_{Y,p}^{\times}
 &=\{(u_i,u_j)\in\cO_{C_i,p}^{\times}\times\cO_{C_j,p}^{\times}
       :u_i(p)=u_j(p)\}.
 \end{aligned}
\]
Thus $\ker q_p=\cO_{Y,p}^{\times}$, and $q_p(1,\lambda)=\lambda$ proves
surjectivity for every $\lambda\in\mathbb C^{\times}$.

Write $H^r(\mathcal F)=H^r(Y,\mathcal F)$. The associated long exact sequence begins
\[
 \begin{aligned}
 1&\longrightarrow H^0(\cO_Y^{\times})\longrightarrow H^0(\mathcal U)
 \xrightarrow{q_*}H^0(\mathcal S)\xrightarrow{\partial}H^1(\cO_Y^{\times})\\
 &\longrightarrow H^1(\mathcal U)\longrightarrow H^1(\mathcal S)
 \longrightarrow H^2(\cO_Y^{\times})\longrightarrow\cdots.
 \end{aligned}
\]
For the component inclusions $j_i\colon C_i\hookrightarrow Y$, direct image is exact,
so $\mathcal U=\prod_i(j_i)_*\cO_{C_i}^{\times}$ gives the following identifications.
Here $H^1(\cO^{\times})=\Pic$ is \resultcite{Tag~09NU}{Stacks}, and
$\Pic(\mathbb P^1)=\mathbb Z$ is \resultcite{Tag~0BXJ}{Stacks}:
\[
 \begin{aligned}
 H^0(\cO_Y^{\times})&=\mathbb C^{\times},
 &H^0(\mathcal U)&=(\mathbb C^{\times})^a,\\
 H^1(\cO_Y^{\times})&=\Pic(Y),
 &H^1(\mathcal U)&=\prod_i\Pic(C_i)\simeq\mathbb Z^a.
 \end{aligned}
\]
Every restriction map of $\mathcal S$ is surjective. Hence $\mathcal S$ is flasque and
\[
 H^0(\mathcal S)=(\mathbb C^{\times})^{a-1},
 \qquad H^r(\mathcal S)=0\quad(r>0).
\]
Substitution yields the exact sequence
\[
 1\longrightarrow\mathbb C^{\times}\xrightarrow{\Delta}(\mathbb C^{\times})^a
 \xrightarrow{\beta}(\mathbb C^{\times})^{a-1}
 \xrightarrow{\partial}\Pic(Y)\xrightarrow{\deg}\mathbb Z^a\longrightarrow0,
\]
where
\[
 \Delta(c)=(c,\ldots,c),
 \qquad\beta((\mu_i))=(\mu_j\mu_i^{-1})_{e:i\to j},
 \qquad\deg\cL=(\deg(\cL|_{C_i}))_i.
\]
The connecting map sends $(\lambda_e)$ to the line bundle obtained by gluing
$\cO_{C_i}$ with fiber identifications $s_j(p_e)=\lambda_e s_i(p_e)$.
For any $(\lambda_e)$, the recursion
\[
 \mu_{\mathrm{root}}=1,
 \qquad\mu_j=\lambda_e\mu_i\quad(e\colon i\to j)
 \quad\Longrightarrow\quad\beta((\mu_i))=(\lambda_e)
\]
is well-defined on the tree. Thus $\beta$ is surjective and exactness gives
$\ker(\deg)=\operatorname{im}\partial=0$; the terminal zero gives surjectivity of $\deg$.
Therefore $\Pic(Y)\xrightarrow{\sim}\mathbb Z^a$.

For passage from $Y$ to $D$, we use the units-sequence argument in
\resultcite{proof of Tag~0C6S}{Stacks}, writing its kernel explicitly by means of the finite logarithm.
Put $\cI_D=\ker(\cO_D\to\cO_Y)$ and choose $r\ge2$ with $\cI_D^r=0$.
The additive restriction sequence and its cohomology sequence give
\[
 0\longrightarrow\cI_D\longrightarrow\cO_D\longrightarrow\cO_Y\longrightarrow0,
\]
\[
 0\longrightarrow H^0(D,\cI_D)\longrightarrow\mathbb C
 \xrightarrow{\id}\mathbb C\longrightarrow H^1(D,\cI_D)
 \longrightarrow H^1(D,\cO_D)=0.
\]
Thus $H^0(D,\cI_D)=H^1(D,\cI_D)=0$, and $H^2(D,\cI_D)=0$ since $\dim D=1$.
The finite logarithm is an isomorphism of abelian sheaves
\[
 \mathcal G_D:=1+\cI_D\xrightarrow[\sim]{\log}\cI_D,
\]
\[
 \log(1+b)=\sum_{k=1}^{r-1}\frac{(-1)^{k+1}b^k}{k},
 \qquad\exp(b)=\sum_{k=0}^{r-1}\frac{b^k}{k!}.
\]
Indeed, the formal identities for logarithm and exponential terminate on the nilpotent ideal;
commutativity gives $\log((1+b)(1+c))=\log(1+b)+\log(1+c)$, and the displayed maps are inverse.
Thus $H^j(D,\mathcal G_D)=0$ for $j=1,2$, and
\[
 1\longrightarrow\mathcal G_D\longrightarrow\cO_D^{\times}
 \longrightarrow\cO_Y^{\times}\longrightarrow1,
\]
\[
 0=H^1(D,\mathcal G_D)\longrightarrow\Pic(D)\longrightarrow\Pic(Y)
 \longrightarrow H^2(D,\mathcal G_D)=0.
\]
Consequently $\Pic(D)\xrightarrow{\sim}\Pic(Y)\xrightarrow{\deg}\mathbb Z^a$.

Choose $z_i\in C_i\setminus\bigcup_{j\ne i}C_j$.
Choose regular parameters $u,v$ in $\cO_{X,z_i}$ with $C_i=(u=0)$.
After shrinking a neighborhood of $z_i$, the zero scheme of $v$ on $D$ is supported only at $z_i$;
gluing its ideal to $\cO_D$ away from $z_i$ defines a subscheme $Z_i\subset D$.
In the completed local ring,
\[
 \widehat{\cO}_{D,z_i}\simeq\mathbb C[[u,v]]/(u^{m_i}),
 \qquad C_i=(u=0),
 \qquad (u^{m_i}):v=(u^{m_i}).
\]
Thus the local equation $v=0$ defines an effective Cartier divisor $Z_i$ on $D$, supported at $z_i$, with
\[
 \deg\bigl(\cO_D(Z_i)|_{C_j}\bigr)=\delta_{ij}.
\]
It follows that
\[
 n_i:=\deg(\cN|_{C_i})\ge0
 \quad\Longrightarrow\quad\cN\simeq\cO_D(Z),
 \qquad Z:=\sum_i n_iZ_i.
\]
Its defining section $s_Z$ is a nonzerodivisor. Hence
\[
 0\longrightarrow\cM\xrightarrow{\cdot s_Z}\cM\otimes\cN
 \longrightarrow(\cM\otimes\cN)|_Z\longrightarrow0,
 \qquad H^1\bigl(Z,(\cM\otimes\cN)|_Z\bigr)=0,
\]
which proves \eqref{eq:3-2}.
\end{proof}

The preceding vanishing statement controls multiplication by a defining section.
Applied to differences of members of an exceptional collection, the next proposition will give
injectivity of the composition maps on $\Ext^1$.

\begin{proposition}\label{prop:3-4}
Suppose $D>0$ is effective and left-orthogonal, $K_X$ is nonnegative on every component of $D$,
and $A,A+D$ are left-orthogonal classes. Multiplication by the defining section of $D$ induces an injection
\[
 H^1(X,\cO_X(A))\longrightarrow H^1(X,\cO_X(A+D)).
\]
The same map is an isomorphism on $H^0$ and a surjection on $H^2$.
\end{proposition}

\begin{proof}
The sequence
\[
 0\longrightarrow\cO_X(-A-D)\longrightarrow\cO_X(-A)
 \longrightarrow\cO_D(-A)\longrightarrow0
\]
shows that $\cO_D(-A)$ is acyclic. Since $\omega_X|_D$ has nonnegative component degrees,
Lemma~\ref{lem:3-3} gives $H^1(D,\cO_D(K_X-A))=0$.
The Cartier divisor $D$ is Gorenstein, with dualizing sheaf $\omega_D=\cO_D(K_X+D)$
by the effective-Cartier adjunction calculation in \resultcite{Tags~0B4B and~0AA4}{Stacks}.
Serre duality on $D$, in the form \resultcite{Tag~0BS2(5)}{Stacks}, yields
\[
 H^0(D,\cO_D(A+D))\simeq H^1(D,\cO_D(K_X-A))^{\vee}=0.
\]
The cohomology sequence of
$0\to\cO_X(A)\to\cO_X(A+D)\to\cO_D(A+D)\to0$ proves all three assertions.
\end{proof}

\subsection{Composition of forward morphisms}\label{subsec:composition}

Let $\cE=(\cL_1,\ldots,\cL_\gamma)$ be an exceptional collection of line bundles satisfying
\begin{equation}\label{eq:3-3}
 \begin{gathered}
 K_XC\ge0\quad\text{for every component $C$ of every effective divisor}\\
 \text{of a nonzero map $\cL_i\longrightarrow\cL_j$, $i<j$.}
 \end{gathered}
\end{equation}
Write $\cL_i=\cO_X(D_i)$ and
\[
 A=\Ext_X^*\!\left(\bigoplus_i\cL_i,\bigoplus_i\cL_i\right),
 \qquad R=A^0,
 \qquad e_jA^qe_i=\Ext^q(\cL_i,\cL_j).
\]
Lemma~\ref{lem:3-2} will identify the connected components of the degree-zero Hom algebra with
linearly ordered chains. Proposition~\ref{prop:3-4} then makes the degree-one composition maps along
these chains injective. To pass from this injectivity to projectivity over $R$, we use the following
decomposition of a chain representation. A complement $W_j$ records the vectors appearing at vertex $j$;
their images at later vertices form copies of the projective representation supported on $j,\ldots,m$.
We record the construction in the following lemmas. The projective modules are those of
\resultcite{\S1, ``Idempotents''~(3)--(5), p.~4}{CB92}; the chain criterion is the field case of
\resultcite{Example~4.2(i) and Theorem~5.4}{LZ11}.

\begin{lemma}[\resultcite{\S1, p.~4}{CB92}]\label{lem:3-5}
Let $Q_m=(1\to2\to\cdots\to m)$, $B=\mathbb C Q_m$, and let $\varepsilon_j$ be the stationary path at $j$.
With paths multiplied in composition order, the left $B$-module $P_j=B\varepsilon_j$ is projective and
\[
 (P_j)_r=
 \begin{cases}
  \mathbb C p_{j,r},&j\le r,\\
  0,&r<j,
 \end{cases}
 \qquad (P_j)_{r\to r+1}(p_{j,r})=p_{j,r+1},
\]
where $p_{j,r}$ is the unique path $j\to r$.
Moreover, $\Hom_B(P_j,V)\simeq V_j$ naturally in $V$.
\end{lemma}

\begin{proof}
We spell out the path-basis description and the idempotent argument of
\resultcite{\S1, ``Idempotents''~(3)--(5), p.~4}{CB92} for the chain $Q_m$.
The path basis gives the displayed representation, and
\[
 1_B=\sum_{j=1}^m\varepsilon_j,
 \qquad B=\bigoplus_{j=1}^m B\varepsilon_j=\bigoplus_{j=1}^m P_j.
\]
Thus each $P_j$ is a direct summand of a free module. The inverse maps
\[
 \begin{aligned}
 \Hom_B(P_j,V)&\longrightarrow V_j,&\varphi&\longmapsto\varphi(\varepsilon_j),\\
 V_j&\longrightarrow\Hom_B(P_j,V),&v&\longmapsto[b\varepsilon_j\longmapsto bv]
 \end{aligned}
\]
give the final assertion.
\end{proof}

\begin{lemma}\label{lem:3-6}
Let $V=(V_1\xrightarrow{f_1}\cdots\xrightarrow{f_{m-1}}V_m)$ be a finite-dimensional representation such that
\[
 \ker f_r=0\qquad\text{for every }r\in\{1,\ldots,m-1\}.
\]
Choose
\[
 W_1=V_1,
 \qquad V_j=f_{j-1}(V_{j-1})\oplus W_j\quad(2\le j\le m),
\]
and put
\[
 F_{rj}=
 \begin{cases}
  f_{r-1}\circ\cdots\circ f_j,&j<r,\\
  \id_{V_j},&j=r.
 \end{cases}
\]
Then
\[
 V_r=\bigoplus_{j\le r}F_{rj}(W_j),
 \qquad\dim W_1=\dim V_1,
 \qquad\dim W_j=\dim V_j-\dim V_{j-1}\quad(j\ge2).
\]
\end{lemma}

\begin{proof}
For every $1\le r<m$, the map $f_r$ is injective and preserves direct sums of subspaces. Hence
\[
 \begin{aligned}
 V_1&=W_1,\\
 V_{r+1}&=f_r(V_r)\oplus W_{r+1}\\
 &=\bigoplus_{j\le r}f_rF_{rj}(W_j)\oplus W_{r+1}
 =\bigoplus_{j\le r+1}F_{r+1,j}(W_j).
 \end{aligned}
\]
Induction proves the decomposition; taking dimensions gives the multiplicities.
\end{proof}

\begin{lemma}[cf.~\resultcite{Example~4.2(i)}{LZ11}]\label{lem:3-7}
In the notation of Lemmas~\ref{lem:3-5} and~\ref{lem:3-6}, there is an isomorphism of representations
\begin{equation}\label{eq:3-4}
 V\simeq\bigoplus_{j=1}^m W_j\otimes_{\mathbb C}P_j
 \simeq\bigoplus_{j=1}^m P_j^{\oplus\dim W_j}.
\end{equation}
In particular, $V$ is projective as a left $B$-module.
\end{lemma}

\begin{proof}
For $m\ge2$, the projectivity assertion specializes the chain criterion in
\resultcite{Example~4.2(i) and Theorem~5.4}{LZ11} to the coefficient field $\mathbb C$;
the case $m=1$ is immediate.
The complements of Lemma~\ref{lem:3-6} give the decomposition as follows.
Set $U=\bigoplus_j W_j\otimes_{\mathbb C}P_j$, with $B$ acting on the second factor.
By Lemma~\ref{lem:3-5},
\[
 U_r=\bigoplus_{j\le r}W_j,
 \qquad\iota_r\colon U_r\longrightarrow U_{r+1},
 \qquad(w_1,\ldots,w_r)\longmapsto(w_1,\ldots,w_r,0).
\]
Lemma~\ref{lem:3-6} gives isomorphisms
\[
 \Phi_r\colon U_r\xrightarrow{\sim}V_r,
 \qquad\Phi_r((w_j)_{j\le r})=\sum_{j\le r}F_{rj}(w_j).
\]
They commute with the arrows:
\[
 \Phi_{r+1}\iota_r((w_j))
 =\sum_{j\le r}F_{r+1,j}(w_j)
 =f_r\!\left(\sum_{j\le r}F_{rj}(w_j)\right)
 =f_r\Phi_r((w_j)).
\]
Thus $\Phi\colon U\xrightarrow{\sim}V$ is $B$-linear, and Lemma~\ref{lem:3-5} proves projectivity.
\end{proof}

We now apply this description to the composition maps in the Ext algebra.
The nesting of effective divisors determines $R$, while multiplication on $H^1$ supplies the
injective arrows needed in Lemma~\ref{lem:3-7}.

\begin{proposition}\label{prop:3-8}
The algebra $R$ is a product of path algebras of linearly oriented chains.
The bimodule $A^1$ is projective both as a left and as a right $R$-module.
\end{proposition}

\begin{proof}
For $i<j$, put $H_{ij}:=\Hom(\cL_i,\cL_j)$ and $\alpha_{ij}:=D_j-D_i$.
Semiorthogonality gives
\[
 H^q(X,\cO_X(-\alpha_{ij}))=\Ext^q(\cL_j,\cL_i)=0\qquad(q\ge0).
\]
Thus every $\alpha_{ij}$ is left-orthogonal. If $H_{ij}\ne0$, choose
\[
 0\ne s_{ij}\in H_{ij}=H^0(X,\cO_X(\alpha_{ij})),
 \qquad C_{ij}:=\operatorname{div}(s_{ij})\in|\alpha_{ij}|.
\]
Here $C_{ij}>0$: otherwise $\cL_i\simeq\cL_j$, contrary to $\Hom(\cL_j,\cL_i)=0$.
By \eqref{eq:3-3} and Lemma~\ref{lem:3-2},
\[
 \dim H_{ij}=h^0(X,\cO_X(C_{ij}))=1.
\]
On vertices $1,\ldots,\gamma$, draw $i\to j$ precisely when $i<j$ and $H_{ij}\ne0$;
let $\Gamma$ be the underlying undirected graph.

For $i<j<k$, any two of the three possible edges imply the third, as follows.

\smallskip
\noindent\textup{(i)} $i\to j$ and $j\to k$.
At the generic point $\eta$ of the integral surface $X$,
\[
 (s_{jk}\circ s_{ij})_\eta=(s_{jk})_\eta\circ(s_{ij})_\eta\ne0,
\]
since both maps are nonzero between one-dimensional $\mathbb C(X)$-vector spaces. Consequently,
\[
 0\ne s_{jk}\circ s_{ij}\in H_{ik}
 \quad\Longrightarrow\quad i\to k,
 \qquad C_{ik}=C_{ij}+C_{jk}.
\]
The divisor equality follows from the one-dimensionality of $H_{ik}$.

\smallskip
\noindent\textup{(ii)} $i\to j$ and $i\to k$.
Set $C_0=C_{ik}$ and $E_0=C_{ij}$. Then
\[
 C_0-E_0\sim D_k-D_j,
 \qquad H^q(X,\cO_X(E_0-C_0))=\Ext^q(\cL_k,\cL_j)=0\quad(q\ge0).
\]
The divisors $C_0,E_0$ are effective and left-orthogonal, and \eqref{eq:3-3} supplies
$K_XC\ge0$ on every component of $E_0$. Lemma~\ref{lem:3-2} gives
\[
 C_{ij}<C_{ik},
 \qquad\operatorname{div}(s_{ik}/s_{ij})=C_{ik}-C_{ij}>0,
\]
whence the quotient is regular and
\[
 0\ne s_{ik}/s_{ij}\in H^0(X,\cO_X(D_k-D_j))=H_{jk}
 \quad\Longrightarrow\quad j\to k.
\]

\smallskip
\noindent\textup{(iii)} $i\to k$ and $j\to k$.
This time set $C_0=C_{ik}$ and $E_0=C_{jk}$. We have
\[
 C_0-E_0\sim D_j-D_i,
 \qquad H^q(X,\cO_X(E_0-C_0))=\Ext^q(\cL_j,\cL_i)=0\quad(q\ge0).
\]
Again \eqref{eq:3-3} and Lemma~\ref{lem:3-2} give
\[
 C_{jk}<C_{ik},
 \qquad\operatorname{div}(s_{ik}/s_{jk})=C_{ik}-C_{jk}>0,
\]
so
\[
 0\ne s_{ik}/s_{jk}\in H^0(X,\cO_X(D_j-D_i))=H_{ij}
 \quad\Longrightarrow\quad i\to j.
\]
Thus, for pairwise distinct vertices,
\[
 \{u,v\},\{v,w\}\in E(\Gamma)\quad\Longrightarrow\quad\{u,w\}\in E(\Gamma).
\]
Let $d_\Gamma$ denote graph distance. If a shortest path $v_0,v_1,\ldots,v_\ell$ had $\ell\ge2$, then
\[
 \{v_0,v_2\}\in E(\Gamma)
 \quad\Longrightarrow\quad d_\Gamma(v_0,v_\ell)\le\ell-1<\ell=d_\Gamma(v_0,v_\ell),
\]
a contradiction. Hence every connected component of $\Gamma$ is complete.

Choose a point outside the finitely many effective Hom divisors and trivialize all $\cL_i$ there.
Normalize each Hom generator to have value one. Products of normalized generators remain normalized.
A component with $m$ ordered vertices consequently gives the path algebra of
\[
 Q_m=(1\longrightarrow2\longrightarrow\cdots\longrightarrow m).
\]
There are no degree-zero maps between different components.
For a component $\Lambda=\{a_1<\cdots<a_m\}$, put
\[
 c_\Lambda=\sum_{r=1}^m e_{a_r},
 \qquad R_\Lambda=c_\Lambda R\simeq\mathbb C Q_m,
 \qquad R=\prod_\Lambda R_\Lambda.
\]
For each source $t$, the left $R_\Lambda$-module $c_\Lambda A^1e_t$ is
\begin{equation}\label{eq:3-5}
 V_r^{\Lambda,t}=\Ext^1(\cL_t,\cL_{a_r}),
 \qquad f_r^{\Lambda,t}(\xi)=s_{a_r,a_{r+1}}\circ\xi.
\end{equation}
For $a<b$ in $\Lambda$, Proposition~\ref{prop:3-4} and semiorthogonality give
\[
 \begin{array}{ll}
 t<a:\quad&\Ext^1(\cL_t,\cL_a)\hookrightarrow\Ext^1(\cL_t,\cL_b),
 \quad A=D_a-D_t,\ D=D_b-D_a,\\[2pt]
 a\le t:\quad&\Ext^1(\cL_t,\cL_a)=0.
 \end{array}
\]
Thus every $f_r^{\Lambda,t}$ is injective. Choose the complements $W_j^{\Lambda,t}$
of Lemma~\ref{lem:3-6}. Lemmas~\ref{lem:3-5} and~\ref{lem:3-7} yield
\[
 c_\Lambda A^1e_t\simeq\bigoplus_{j=1}^m W_j^{\Lambda,t}\otimes_{\mathbb C}R_\Lambda e_{a_j},
 \qquad R_\Lambda e_{a_j}=Re_{a_j},
\]
and therefore
\[
 A^1=\bigoplus_{\Lambda,t}c_\Lambda A^1e_t
 \simeq\bigoplus_{\Lambda,t,j}(Re_{a_j})^{\oplus\dim W_j^{\Lambda,t}}.
\]
This proves left projectivity.

For a fixed target $t$, right multiplication acts by precomposition:
\[
 \Ext^1(\cL_b,\cL_t)\longrightarrow\Ext^1(\cL_a,\cL_t),
 \qquad\xi\longmapsto\xi\circ s_{ab}.
\]
Again,
\[
 \begin{array}{ll}
 b<t:\quad&\text{the map is injective by Proposition~\ref{prop:3-4},}
 \quad A=D_t-D_b,\ D=D_b-D_a,\\[2pt]
 t\le b:\quad&\Ext^1(\cL_b,\cL_t)=0.
 \end{array}
\]
Each component of $R^{\mathrm{op}}$ is a chain with the reversed vertex order.
Applying Lemma~\ref{lem:3-7} to these chains proves right projectivity.
\end{proof}

The height estimate uses the left projectivity in Proposition~\ref{prop:3-8}.

\subsection{The first bar differential}\label{subsec:bar-height}

For the exceptional collection under consideration, define
\[
 A=\Ext_X^*\!\left(\bigoplus_i\cL_i,\bigoplus_i\cL_i\right)
   =A^0\oplus A^1\oplus A^2,
 \qquad e_jA^qe_i=\Ext_X^q(\cL_i,\cL_j).
\]
The product is Yoneda composition, and $e_i=\id_{\cL_i}$ denotes the
corresponding summand projection. Thus
\[
 e_ie_j=\delta_{ij}e_i,\qquad \sum_i e_i=1_A,\qquad
 e_iAe_i=\mathbb C e_i,\qquad e_jAe_i=0\quad(j<i).
\]
A \emph{finite graded directed algebra} is a finite-dimensional unital
graded $\mathbb C$-algebra with such ordered idempotents. Its diagonal
algebra and off-diagonal ideal are
\begin{equation}\label{eq:3-6}
 S_0=\bigoplus_i\mathbb C e_i\simeq\mathbb C^\gamma,
 \qquad J=\bigoplus_{i<j}e_jAe_i,\qquad A=S_0\oplus J.
\end{equation}
Write
\begin{equation}\label{eq:3-7}
 R=A^0,\qquad J_0=J\cap R,\qquad R=S_0\oplus J_0,
 \qquad J=J_0\oplus A^1\oplus A^2.
\end{equation}
In particular, $J$ includes the forward degree-zero morphisms. We use
the cohomology algebra $A=H^*(\mathcal A)$.

For a graded $A$-bimodule $T$, put
\begin{equation}\label{eq:3-8}
\begin{gathered}
 C_p(A,T)=\tr_{S_0}\!\left(T\otimes_{S_0}J^{\otimes_{S_0}p}\right),
 \qquad C(A,T)=\bigoplus_{p\ge0}C_p(A,T)[p],\\
 \tr_{S_0}(M)=\bigoplus_i e_iMe_i,\qquad J^{\otimes_{S_0}0}=S_0.
\end{gathered}
\end{equation}
Write $J_d:=J\cap A^d$; thus $J_1=A^1$, $J_2=A^2$, and $J_d=0$
for $d\notin\{0,1,2\}$. For a homogeneous tensor
\[
 z=x\otimes a_1\otimes\cdots\otimes a_p,
 \qquad x\in T^t,\quad a_\nu\in J_{d_\nu},
\]
its internal degree and its degree in $C(A,T)$ are
\begin{equation}\label{eq:3-9}
 D(z)=t+\sum_{\nu=1}^p d_\nu,\qquad
 \deg_C(z)=D(z)-p.
\end{equation}
In the bar formulas of this subsection, unlabeled tensor products are
over $S_0$. We use unshifted tensor factors and apply $[p]$ to the whole
length-$p$ summand. The cyclic interchange therefore uses the internal
degrees $t,d_\nu$. The shift and Koszul conventions
are those of \resultcite{\S2.1}{Kel06}; the cyclic bar differential
recalled in \resultcite{\S5.3}{Kel06} is written below with a
bimodule coefficient and this factor order.

For $p\ge1$, define the face maps
$b_\nu^{(p)}:C_p(A,T)\to C_{p-1}(A,T)$ by
\begin{equation}\label{eq:3-10}
\begin{aligned}
 b_0^{(p)}(z)&=(xa_1)\otimes a_2\otimes\cdots\otimes a_p,\\
 b_\nu^{(p)}(z)&=x\otimes a_1\otimes\cdots\otimes
       (a_\nu a_{\nu+1})\otimes\cdots\otimes a_p
       &&(1\le\nu<p),\\
 b_p^{(p)}(z)&=(-1)^{\kappa(z)}(a_px)\otimes a_1\otimes\cdots\otimes a_{p-1},\\
 \kappa(z)&=d_p\left(t+\sum_{\nu=1}^{p-1}d_\nu\right).
\end{aligned}
\end{equation}
Here $xa_1$ is the right $A$-action on $T$, whereas $a_px$ is the left
action. The sign $(-1)^{\kappa(z)}$ moves $a_p$ past
$x,a_1,\ldots,a_{p-1}$. The superscript $(p)$ records the input bar
length. The differential is
\begin{equation}\label{eq:3-11}
 b\big|_{C_p(A,T)[p]}=\sum_{\nu=0}^p(-1)^\nu b_\nu^{(p)}
 \quad(p\ge1),\qquad b\big|_{C_0(A,T)}=0.
\end{equation}
Thus $(-1)^\nu$ is the alternating bar sign, and the last face also
contains the Koszul sign in \eqref{eq:3-10}. Both $A$ and $T$ have
zero internal differential. For example,
\[
\begin{aligned}
 b(x\otimes a)&=xa-(-1)^{|a||x|}ax,\\
 b(x\otimes a_1\otimes a_2)
  &=xa_1\otimes a_2-x\otimes a_1a_2
       +(-1)^{d_2(t+d_1)}a_2x\otimes a_1.
\end{aligned}
\]
Associativity and the bimodule identities give the face relations
\[
 b_i^{(p-1)}b_j^{(p)}=b_{j-1}^{(p-1)}b_i^{(p)}
 \qquad(0\le i<j\le p,\ p\ge2).
\]
The two occurrences have opposite alternating signs in $b^2$, so
$b^2=0$. Moreover,
\[
 b\bigl(C_p(A,T)^D\bigr)\subseteq C_{p-1}(A,T)^D,
 \qquad D-(p-1)=(D-p)+1,
\]
where the superscript $D$ denotes internal degree. Hence $b$ has
cohomological degree $+1$.

The endpoint order agrees with the bimodule convention. On a summand
indexed by $i_0<\cdots<i_p$,
\[
 x\in e_{i_0}T^t e_{i_p},\qquad
 a_\nu\in e_{i_{p-\nu+1}}J_{d_\nu}e_{i_{p-\nu}}
 \quad(1\le\nu\le p),
\]
and consequently
\[
 xa_1\in e_{i_0}T^{t+d_1}e_{i_{p-1}},\qquad
 a_px\in e_{i_1}T^{t+d_p}e_{i_p}.
\]
For the normal bimodule $T=H^*(\mathcal T)$ these actions are
\[
 xa_1=x\circ a_1,\qquad a_px=S_X^{-1}(a_p)\circ x;
\]
see \eqref{eq:2-4} and \resultcite{Proposition~3.7}{Kuz15}.
Products $a_\nu a_{\nu+1}$ in \eqref{eq:3-10} are likewise in
composition order, $a_\nu\circ a_{\nu+1}$.

We use $R^e=R^{\mathrm{op}}\otimes_{\mathbb C}R$ and
\[
 K\in\mathrm D^{\ge d}\ \Longleftrightarrow\ H^u(K)=0\ (u<d),
 \qquad H^u(K[s])=H^{u+s}(K),\qquad d_{K[s]}=(-1)^s d_K.
\]
For a condition $P$, the indicator $\mathbf1_P$ equals $1$ when $P$
holds and $0$ otherwise. The following lemmas separate the filtration,
the resolutions, and the degree estimate.

\begin{lemma}\label{lem:3-9}
For $\mathbf d=(d_1,\ldots,d_p)\in\{0,1,2\}^p$, put
\[
 \rho(\mathbf d)=\sum_{\nu=1}^p\mathbf1_{d_\nu>0},\qquad
 \rho(x\otimes a_1\otimes\cdots\otimes a_p)=\rho(\mathbf d)
 \quad(a_\nu\in J_{d_\nu}).
\]
For $C=C(A,T)$ and $r\ge0$, define
\begin{equation}\label{eq:3-12}
 F_rC=\bigoplus_{p\ge0}\ 
       \bigoplus_{\substack{\mathbf d\in\{0,1,2\}^p\\\rho(\mathbf d)\le r}}
       \tr_{S_0}\!\left(T\otimes J_{d_1}\otimes\cdots\otimes J_{d_p}\right)[p],
 \qquad F_{-1}C=0.
\end{equation}
These are subcomplexes, and, if $\gamma$ is the number of vertex idempotents,
\[
 0=F_{-1}C\subseteq F_0C\subseteq\cdots\subseteq F_{\gamma-1}C=C.
\]
The associated-graded differential consists exactly of the nonzero
faces involving at least one degree-zero algebra factor. In particular,
\begin{equation}\label{eq:3-13}
 \operatorname{gr}_r^F C:=F_rC/F_{r-1}C
   =\bigoplus_{t\in\mathbb Z}\ \bigoplus_{\mathbf q\in\{1,2\}^r}G_{t,\mathbf q},
\end{equation}
where $G_{t,\mathbf q}$ fixes the internal degree $t$ of $x$ and the
ordered list $\mathbf q=(q_1,\ldots,q_r)$ of positive degrees of the
algebra factors.
\end{lemma}

\begin{proof}
Fix $p\ge1$ and a homogeneous tensor
$z=x\otimes a_1\otimes\cdots\otimes a_p$, with $|x|=t$ and
$|a_\nu|=d_\nu$. The endpoint maps $b_0^{(p)}$ and $b_p^{(p)}$ are
the right and cyclic left actions of \eqref{eq:3-10}; the maps
$b_\nu^{(p)}$, $1\le\nu<p$, are its adjacent products. Their nonzero
outputs have multidegrees
\[
\begin{aligned}
 b_0^{(p)}(z)&:\ (t+d_1;d_2,\ldots,d_p),\\
 b_p^{(p)}(z)&:\ (t+d_p;d_1,\ldots,d_{p-1}),\\
 b_\nu^{(p)}(z)&:\ (t;d_1,\ldots,d_{\nu-1},d_\nu+d_{\nu+1},
                           d_{\nu+2},\ldots,d_p).
\end{aligned}
\]
Here the entry before the semicolon is the coefficient degree; the
remaining entries are the algebra-factor degrees. Therefore, for every
nonzero face,
\begin{equation}\label{eq:3-14}
\begin{aligned}
 \rho(b_0^{(p)}z)&=\rho(z)-\mathbf1_{d_1>0},\\
 \rho(b_p^{(p)}z)&=\rho(z)-\mathbf1_{d_p>0},\\
 \rho(b_\nu^{(p)}z)&=\rho(z)-\mathbf1_{d_\nu>0}
                     -\mathbf1_{d_{\nu+1}>0}
                     +\mathbf1_{d_\nu+d_{\nu+1}>0}\\
  &=\rho(z)-\mathbf1_{d_\nu>0,\ d_{\nu+1}>0}
       \qquad(1\le\nu<p).
\end{aligned}
\end{equation}
If $d_\nu+d_{\nu+1}>2$, the corresponding product is zero because
$A^{>2}=0$, and it contributes nothing. Thus every face weakly
decreases $\rho$, and \eqref{eq:3-11} yields $b(F_rC)\subseteq F_rC$.

For $\rho(z)=r$, write $[z]_r=z+F_{r-1}C$ and denote the induced
differential on $\operatorname{gr}_r^F C$ by $\bar b$.
Formula~\eqref{eq:3-14} gives
\begin{equation}\label{eq:3-15}
\begin{aligned}
 \bar b[z]_r={}&\mathbf1_{d_1=0}[b_0^{(p)}z]_r\\
 &+\sum_{\nu=1}^{p-1}(-1)^\nu
       \mathbf1_{\{d_\nu=0\ \mathrm{or}\ d_{\nu+1}=0\}}
       [b_\nu^{(p)}z]_r\\
 &+(-1)^p\mathbf1_{d_p=0}[b_p^{(p)}z]_r.
\end{aligned}
\end{equation}
Empty sums are zero; in bar length zero, $\bar b=0$. The last face
survives only when $d_p=0$, in which case $\kappa(z)=0$. A surviving
endpoint leaves $t$ unchanged. A surviving adjacent product has degrees
\[
 (0,0)\longmapsto0,\qquad (0,q)\longmapsto q,\qquad
 (q,0)\longmapsto q\qquad(q\in\{1,2\}).
\]
In particular, if
\[
 \{\nu:d_\nu>0\}=\{\nu_1<\cdots<\nu_r\},\qquad
 \mathbf q(z)=(d_{\nu_1},\ldots,d_{\nu_r}),
\]
then every surviving face preserves $(t,\mathbf q(z))$. The list is
empty for $r=0$. This proves the decomposition into the subcomplexes
$G_{t,\mathbf q}$ in \eqref{eq:3-13}.

Finally, for $p\ge1$, the idempotent decomposition is
\[
 J^{\otimes_{S_0}p}\simeq
 \bigoplus_{i_0<\cdots<i_p}
 e_{i_p}Je_{i_{p-1}}\otimes_{\mathbb C}\cdots
           \otimes_{\mathbb C}e_{i_1}Je_{i_0}.
\]
There is no such index sequence for $p\ge \gamma$. Hence
\[
 J^{\otimes_{S_0}p}=0\quad(p\ge \gamma),\qquad
 \rho(z)\le p\le \gamma-1,
\]
which proves finiteness.
\end{proof}

Fixing $(t,\mathbf q)$ leaves only the degree-zero factors free to
vary. They occur in the gaps between the fixed positive-degree
factors. The first $r$ gaps form the bar resolutions computing
successive derived tensor products over $R$; the remaining gap closes
the word and computes the derived trace over $R^e$. The next lemma
identifies these complexes, including the shift restoring the original
internal degrees and the sign change comparing the two total
differentials.

\begin{lemma}\label{lem:3-10}
Fix $t\in\mathbb Z$ and $\mathbf q=(q_1,\ldots,q_r)\in\{1,2\}^r$,
and put $D=t+\sum_iq_i$. Regard $T^t$ and $A^{q_i}$ as ordinary
bimodules concentrated in degree zero. Define
\begin{equation}\label{eq:3-16}
 K_{t,\mathbf q}
  =\left(T^t\otimes_R^{\mathbf L}A^{q_1}
      \otimes_R^{\mathbf L}\cdots\otimes_R^{\mathbf L}A^{q_r}\right)
      \otimes_{R^e}^{\mathbf L}R.
\end{equation}
Then
\begin{equation}\label{eq:3-17}
 G_{t,\mathbf q}\simeq K_{t,\mathbf q}[r-D].
\end{equation}
For $r=0$, this reads
$G_{t,\varnothing}\simeq(T^t\otimes_{R^e}^{\mathbf L}R)[-t]$.
\end{lemma}

\begin{proof}
We first recall the relative bar resolution discussed after
\resultcite{Definition~2.3}{CLMS19}. We then identify its iterated
totalization with $G_{t,\mathbf q}$, keeping track of the internal
degrees and the signs in \eqref{eq:3-11}.

Let $\epsilon:R\to S_0$ be the projection for $R=S_0\oplus J_0$,
and write $\bar u=u-\epsilon(u)\in J_0$. Since $J_0$ is a two-sided
ideal, $\epsilon$ is an algebra homomorphism and $u\mapsto\bar u$
is an $S_0$-bimodule map. All tensor products without a subscript in
this proof are over $S_0$, unless $\mathbb C$ is explicitly indicated.

For $h\ge0$, define the $R$-bimodule
\[
 B_h=R\otimes J_0^{\otimes h}\otimes R,\qquad
 J_0^{\otimes0}=S_0,
\]
and use the notation
\[
 u[c_1|\cdots|c_h]v:=u\otimes c_1\otimes\cdots\otimes c_h\otimes v,
 \qquad u,v\in R,\quad c_\ell\in J_0.
\]
In particular, $u[\,]v=u\otimes v\in B_0$. For $h\ge1$, define maps
$d_{h,\ell}:B_h\to B_{h-1}$, $0\le\ell\le h$, by
\begin{equation}\label{eq:3-18}
\begin{aligned}
 d_{h,0}\bigl(u[c_1|\cdots|c_h]v\bigr)
    &=uc_1[c_2|\cdots|c_h]v,\\
 d_{h,\ell}\bigl(u[c_1|\cdots|c_h]v\bigr)
    &=u[c_1|\cdots|c_\ell c_{\ell+1}|\cdots|c_h]v
        \quad(1\le\ell<h),\\
 d_{h,h}\bigl(u[c_1|\cdots|c_h]v\bigr)
    &=u[c_1|\cdots|c_{h-1}]c_hv.
\end{aligned}
\end{equation}
Set
\[
 d_h=\sum_{\ell=0}^h(-1)^\ell d_{h,\ell},\qquad
 \mu:B_0\to R,\qquad \mu(u[\,]v)=uv.
\]
Products of two elements of $J_0$ remain in $J_0$, so the formulas
are well-defined. Associativity gives $d_hd_{h+1}=0$ and $\mu d_1=0$.
Thus
\begin{equation}\label{eq:3-19}
 \cdots\xrightarrow{d_3}B_2\xrightarrow{d_2}B_1
       \xrightarrow{d_1}B_0\xrightarrow{\mu}R\longrightarrow0
\end{equation}
is an augmented complex of $R$-bimodules.

To explain the normalization, put
$\widetilde B_h=R\otimes R^{\otimes h}\otimes R$. Let
$D_h\subseteq\widetilde B_h$ be the sum of the tensor subspaces
having at least one of their $h$ middle factors in $S_0$; set $D_0=0$.
These are exactly the degenerate tensors, obtained by inserting
$1\in S_0$ and using $S_0$-balancing. The two faces removing such a
factor cancel, while all other faces retain an $S_0$-factor. Hence
$D_\bullet$ is a subcomplex of the unnormalized bar complex, and
\[
 \widetilde B_h/D_h\simeq
 R\otimes(R/S_0)^{\otimes h}\otimes R
 \xrightarrow[\ R/S_0\simeq J_0\ ]{\ \sim\ }B_h.
\]
The induced differential is \eqref{eq:3-18} with its alternating
signs. This is the normalized relative bar complex. The description
as a quotient by degenerate tensors is the normalization construction
of \resultcite{Lemma~14.23.6 and Remark~14.23.7, Tags 019A and 0FKI}{Stacks},
applied to the relative bar simplicial bimodule.

For exactness, we use the $S_0$-relative version of the standard
unit-insertion contraction recorded in \resultcite{\S2, p.~4}{IIVZ14}.
Under $R/S_0\simeq J_0$, insertion of the first factor becomes
insertion of $\bar u$. Define
\begin{equation}\label{eq:3-20}
\begin{aligned}
 s_{-1}:R&\longrightarrow B_0,
      &s_{-1}(v)&=1[\,]v,\\
 s_h:B_h&\longrightarrow B_{h+1},
      &s_h\bigl(u[c_1|\cdots|c_h]v\bigr)
        &=1[\bar u|c_1|\cdots|c_h]v\qquad(h\ge0).
\end{aligned}
\end{equation}
These maps are $S_0$-linear on the left and $R$-linear on the right.
For $h=0$,
\[
\begin{aligned}
 (d_1s_0+s_{-1}\mu)(u[\,]v)
   &=\bar u[\,]v-1[\,]\bar uv+1[\,]uv\\
   &=\bar u[\,]v+1[\,]\epsilon(u)v=u[\,]v.
\end{aligned}
\]
For $h\ge1$, the adjacent-product and final-action terms cancel in
pairs. Since $\overline{uc_1}=uc_1$, the remaining terms are
\[
\begin{aligned}
 &(d_{h+1}s_h+s_{h-1}d_h)\bigl(u[c_1|\cdots|c_h]v\bigr)\\
 &\qquad=\bar u[c_1|\cdots|c_h]v
          +1[\epsilon(u)c_1|c_2|\cdots|c_h]v\\
 &\qquad=u[c_1|\cdots|c_h]v.
\end{aligned}
\]
Together with $\mu s_{-1}=\id_R$, this proves exactness, split as a
complex of right $R$-modules.

For projectivity, use $R^e=R^{\mathrm{op}}\otimes_{\mathbb C}R$
with left action $(a^{\mathrm{op}}\otimes b)z=bza$ on an
$R$-bimodule. Decomposition by the vertex idempotents gives
\begin{equation}\label{eq:3-21}
\begin{gathered}
 B_h\simeq\bigoplus_{i,j}
 Re_i\otimes_{\mathbb C}(e_iJ_0^{\otimes h}e_j)
        \otimes_{\mathbb C}e_jR,\\
 Re_i\otimes_{\mathbb C}e_jR
      \simeq R^e(e_j^{\mathrm{op}}\otimes e_i).
\end{gathered}
\end{equation}
Every $B_h$ is therefore a finite direct sum of projective
$R^e$-modules. Write $B_\bullet=\bigoplus_{h\ge0}B_h[h]$ for the
unaugmented complex, with $B_h$ in cohomological degree $-h$ and
differential $d_h$ for $h\ge1$. We have proved
\[
 \mu:B_\bullet\longrightarrow R
 \quad\text{is a projective bimodule resolution.}
\]

For an ordinary left $R$-module $V$, tensoring the right $R$-linear
contraction with $V$ proves that $B_\bullet\otimes_RV\to V$ is exact.
Its terms are projective left $R$-modules, because
\[
 B_h\otimes_RV\simeq R\otimes J_0^{\otimes h}\otimes V
 \simeq\bigoplus_i Re_i\otimes_{\mathbb C}
                    (e_iJ_0^{\otimes h}\otimes V).
\]
Consequently, for a right $R$-module $M$, define
\begin{equation}\label{eq:3-22}
\begin{aligned}
 \operatorname{B}_R(M,V)&:=M\otimes_RB_\bullet\otimes_RV\\
   &\simeq\bigoplus_{h\ge0}
           (M\otimes J_0^{\otimes h}\otimes V)[h]
     \simeq M\otimes_R^{\mathbf L}V.
\end{aligned}
\end{equation}
The faces are \eqref{eq:3-18}, with the outer products replaced by
the right action on $M$ and the left action on $V$. For an
$R$-bimodule $P$, regard $P$ as a right $R^e$-module via
$p(a^{\mathrm{op}}\otimes b)=apb$. Then
\[
\begin{aligned}
 P\otimes_{R^e}B_h
   &\simeq\bigoplus_{i,j}
         e_jPe_i\otimes_{\mathbb C}e_iJ_0^{\otimes h}e_j\\
   &\simeq\tr_{S_0}(P\otimes J_0^{\otimes h}).
\end{aligned}
\]
With the projectivity just proved, the relative Tor construction
of \resultcite{Definition~2.3}{CLMS19} is the ordinary derived trace.
Its relative Hochschild complex is
\begin{equation}\label{eq:3-23}
 \operatorname{CH}_R(P):=P\otimes_{R^e}B_\bullet
 \simeq\bigoplus_{h\ge0}\tr_{S_0}(P\otimes J_0^{\otimes h})[h]
 \simeq P\otimes_{R^e}^{\mathbf L}R.
\end{equation}
Its last face acts on the left of $P$, as in \eqref{eq:3-10} with
all internal degrees zero. These constructions also apply to bounded
complexes by totalization. Indeed, $J_0^{\otimes h}=0$ for $h\ge \gamma$
by directedness, so the bar directions are finite and the resolutions
are bounded complexes of flat modules on the sides used for tensoring.

We now apply these resolutions to the summand $G_{t,\mathbf q}$. Put
\[
 M_0=T^t,\qquad M_j=A^{q_j}\quad(1\le j\le r),\qquad
 \sigma=r-D=r-t-\sum_{j=1}^r q_j,
\]
with each $M_j$ temporarily concentrated in degree zero. For
$\mathbf h=(h_0,\ldots,h_r)\in\mathbb Z_{\ge0}^{r+1}$, define
\[
 |\mathbf h|=\sum_{k=0}^r h_k,\qquad
 H_k(\mathbf h)=\sum_{j=0}^{k-1}h_j\quad(H_0=0),
\]
and let $\mathbf e_k$ denote the $k$th coordinate vector in
$\mathbb Z^{r+1}$, with coordinates numbered $0,\ldots,r$. Define
the ordinary vector space
\begin{equation}\label{eq:3-24}
 U_{\mathbf h}=\tr_{S_0}\!\left(
 M_0\otimes J_0^{\otimes h_0}\otimes M_1\otimes J_0^{\otimes h_1}
       \otimes\cdots\otimes M_r\otimes J_0^{\otimes h_r}\right).
\end{equation}
Here $h_k$ counts the degree-zero factors after $M_k$; the gap after
$M_r$ ends cyclically at $M_0$. For $r=0$, \eqref{eq:3-24} means
simply $U_{(h_0)}=\tr_{S_0}(M_0\otimes J_0^{\otimes h_0})$.

A length-$p$ tensor in $G_{t,\mathbf q}$ has exactly $r$ positive
algebra factors. Let
$n_1<\cdots<n_r$ be their positions, and set $n_0=0$, $n_{r+1}=p+1$.
The gap lengths and the inverse reconstruction are
\begin{equation}\label{eq:3-25}
 h_k=n_{k+1}-n_k-1\quad(0\le k\le r),\qquad
 n_j=j+H_j(\mathbf h)\quad(1\le j\le r),\qquad
 p=r+|\mathbf h|.
\end{equation}
Equivalently, the ordered algebra degrees are uniquely
\[
 (\underbrace{0,\ldots,0}_{h_0},q_1,
  \underbrace{0,\ldots,0}_{h_1},q_2,\ldots,q_r,
  \underbrace{0,\ldots,0}_{h_r}).
\]
For $r=0$, there are no positive positions and $h_0=p$. The
decomposition $J=J_0\oplus A^1\oplus A^2$ is direct, so different
multidegree words give distinct direct summands. Formula~\eqref{eq:3-25}
identifies each of them with precisely one $U_{\mathbf h}$. Its
original internal degree is $D$, hence
\[
 \deg_C=D-p=D-r-|\mathbf h|=-(|\mathbf h|+\sigma).
\]
Since $U_{\mathbf h}$ is placed in degree zero, this proves
\begin{equation}\label{eq:3-26}
 G_{t,\mathbf q}=\bigoplus_{\mathbf h\in\mathbb Z_{\ge0}^{r+1}}
             U_{\mathbf h}[|\mathbf h|+\sigma],\qquad
 G_{t,\mathbf q}^{u}=\bigoplus_{|\mathbf h|=D-r-u}U_{\mathbf h}.
\end{equation}
In these equalities of graded vector spaces, the differential
decreases one gap length.

We next define its individual gap faces. For $m_j\in M_j$ and
$c_{k,\ell}\in J_0$, let
$\mathbf c_k=(c_{k,1},\ldots,c_{k,h_k})$ and abbreviate a tensor by
\[
\begin{aligned}
 \zeta&=\langle m_0;\mathbf c_0;m_1;\mathbf c_1;\ldots;m_r;\mathbf c_r\rangle\\
   &=m_0\otimes c_{0,1}\otimes\cdots\otimes c_{0,h_0}
      \otimes m_1\otimes\cdots\otimes m_r
      \otimes c_{r,1}\otimes\cdots\otimes c_{r,h_r}.
\end{aligned}
\]
The brackets denote the displayed tensor in $\tr_{S_0}$. An empty
list $\mathbf c_k$ contributes no factor. For $h_k>0$ and
$0\le\ell\le h_k$, define the unsigned map
\[
 \partial_{k,\ell}^{\mathbf h}:
 U_{\mathbf h}\longrightarrow U_{\mathbf h-\mathbf e_k}
\]
as follows. The index $k$ chooses the gap, and $\ell$ chooses a face
within that gap. The first face uses the right action on $M_k$:
\begin{equation}\label{eq:3-27}
 \partial_{k,0}^{\mathbf h}(\zeta)
   =\langle\ldots;m_kc_{k,1};(c_{k,2},\ldots,c_{k,h_k});\ldots\rangle.
\end{equation}
For $1\le\ell<h_k$, the middle face multiplies two adjacent elements
of $J_0$:
\begin{equation}\label{eq:3-28}
 \partial_{k,\ell}^{\mathbf h}(\zeta)
 =\langle\ldots;m_k;
    (c_{k,1},\ldots,c_{k,\ell}c_{k,\ell+1},\ldots,c_{k,h_k});
     \ldots\rangle.
\end{equation}
For $k<r$, the last face uses the left action on $M_{k+1}$:
\begin{equation}\label{eq:3-29}
 \partial_{k,h_k}^{\mathbf h}(\zeta)
 =\langle\ldots;m_k;(c_{k,1},\ldots,c_{k,h_k-1});
          c_{k,h_k}m_{k+1};\mathbf c_{k+1};\ldots\rangle.
\end{equation}
All entries not displayed as changed remain in their original order.
For $k=r$, the last face instead acts on the initial coefficient:
\begin{equation}\label{eq:3-30}
 \partial_{r,h_r}^{\mathbf h}(\zeta)
 =\langle c_{r,h_r}m_0;\mathbf c_0;m_1;\mathbf c_1;\ldots;
          m_r;(c_{r,1},\ldots,c_{r,h_r-1})\rangle
 \quad(r\ge1).
\end{equation}
For the single-gap case $r=0$, the same rule reads
\[
 \partial_{0,h_0}^{(h_0)}
       \langle m_0;(c_{0,1},\ldots,c_{0,h_0})\rangle
 =\langle c_{0,h_0}m_0;(c_{0,1},\ldots,c_{0,h_0-1})\rangle.
\]
The cyclic face has no Koszul sign because $|c_{r,h_r}|=0$. Each
face removes exactly one $J_0$-factor, which explains its target
$U_{\mathbf h-\mathbf e_k}$. All maps are $S_0$-balanced; zero
products or zero module actions give the zero tensor.

Set
\begin{equation}\label{eq:3-31}
 \partial_k^{\mathbf h}:=\sum_{\ell=0}^{h_k}
          (-1)^\ell\partial_{k,\ell}^{\mathbf h}\quad(h_k>0),
 \qquad \partial_k^{\mathbf h}=0\quad(h_k=0).
\end{equation}
We omit the superscript when the source $U_{\mathbf h}$ is fixed.
For example, if $r=1$, $\mathbf h=(1,1)$, and
$\zeta=m_0\otimes c\otimes m_1\otimes d$ with $c,d\in J_0$, then
\[
\begin{aligned}
 \partial_{0,0}(\zeta)&=m_0c\otimes m_1\otimes d,
 &\partial_{0,1}(\zeta)&=m_0\otimes cm_1\otimes d,\\
 \partial_{1,0}(\zeta)&=m_0\otimes c\otimes m_1d,
 &\partial_{1,1}(\zeta)&=dm_0\otimes c\otimes m_1.
\end{aligned}
\]
In particular, $\partial_0=\partial_{0,0}-\partial_{0,1}$ and
$\partial_1=\partial_{1,0}-\partial_{1,1}$.

In the full bar word, the face $\partial_{k,\ell}$ has index
$k+H_k(\mathbf h)+\ell$: there are $k$ positive factors and
$H_k(\mathbf h)$ degree-zero factors before this gap.
Formula~\eqref{eq:3-15} therefore gives
\begin{equation}\label{eq:3-32}
 \bar b\big|_{U_{\mathbf h}}
 =\sum_{\substack{0\le k\le r\\h_k>0}}
        \sum_{\ell=0}^{h_k}(-1)^{k+H_k(\mathbf h)+\ell}
                                  \partial_{k,\ell}^{\mathbf h}
 =\sum_{k=0}^r(-1)^{k+H_k(\mathbf h)}\partial_k^{\mathbf h}.
\end{equation}
The faces involving only positive factors are absent by
Lemma~\ref{lem:3-9}.

For comparison with derived tensor products, form the ordinary
iterated bar complex of $R$-bimodules
\[
\begin{aligned}
 H_{t,\mathbf q}
   &:=M_0\otimes_R B_\bullet\otimes_RM_1
        \otimes_R\cdots\otimes_R B_\bullet\otimes_RM_r\\
   &\simeq M_0\otimes_R^{\mathbf L}M_1
        \otimes_R^{\mathbf L}\cdots\otimes_R^{\mathbf L}M_r.
\end{aligned}
\]
There is one copy of $B_\bullet$ in each of the first $r$ gaps;
for $r=0$, set $H_{t,\varnothing}=M_0$. The quasi-isomorphism follows
by applying \eqref{eq:3-22} successively, retaining the outer
$R$-actions. More explicitly, each
$B_\bullet\otimes_RM_j\to M_j$ is a quasi-isomorphism of bimodule
complexes whose source is a bounded complex of projective left
$R$-modules. Tensoring against the preceding complex therefore
computes the derived tensor product. Closing the remaining gap by
\eqref{eq:3-23} gives
\begin{equation}\label{eq:3-33}
\begin{gathered}
 B_{t,\mathbf q}:=H_{t,\mathbf q}\otimes_{R^e}B_\bullet
        \simeq K_{t,\mathbf q},\\
 B_{t,\mathbf q}=\bigoplus_{\mathbf h}U_{\mathbf h}[|\mathbf h|],
 \qquad
 \delta\big|_{U_{\mathbf h}}
       =\sum_{k=0}^r(-1)^{H_k(\mathbf h)}\partial_k^{\mathbf h}.
\end{gathered}
\end{equation}
Indeed, the preceding bar directions have total degree
$-H_k(\mathbf h)$. Their contribution to the tensor-product sign is
$(-1)^{-H_k(\mathbf h)}=(-1)^{H_k(\mathbf h)}$. The local identities
$\partial_k^2=0$ and $\partial_j\partial_k=\partial_k\partial_j$
for $j\ne k$ follow from associativity and the commuting left and
right actions on each $M_j$; with the displayed total signs they
imply $\delta^2=0$.

Finally, define an integer and a degree-zero map by
\[
 \eta(\mathbf h)=\sum_{k=0}^r(k+\sigma)h_k,\qquad
 \Phi\big|_{U_{\mathbf h}}=(-1)^{\eta(\mathbf h)}\id:
 G_{t,\mathbf q}\longrightarrow B_{t,\mathbf q}[\sigma].
\]
The source and target degrees agree by \eqref{eq:3-26}. For $h_k>0$,
\[
 \eta(\mathbf h-\mathbf e_k)+k+H_k(\mathbf h)
 \equiv\eta(\mathbf h)+\sigma+H_k(\mathbf h)\pmod2.
\]
Since $d_{B[\sigma]}=(-1)^\sigma\delta$, formulas
\eqref{eq:3-32} and \eqref{eq:3-33} yield
\[
 \Phi\bar b=(-1)^\sigma\delta\Phi,\qquad
 G_{t,\mathbf q}\xrightarrow[\Phi]{\ \sim\ }B_{t,\mathbf q}[\sigma]
       \simeq K_{t,\mathbf q}[r-D].
\]
This proves the assertion, including $r=0$.
\end{proof}

\begin{lemma}[{\resultcite{\S1, pp.~7--8}{CB92}}]\label{lem:3-11}
Suppose $R$ is a product of path algebras of finite acyclic quivers.
Every left $R$-module has projective dimension at most one. For an
ordinary left $R$-module $N$ and a complex
$M\in\mathrm D^{\ge d}$ of right $R$-modules,
\begin{equation}\label{eq:3-34}
 M\otimes_R^{\mathbf L}N\in\mathrm D^{\ge d-1}.
\end{equation}
If $N$ is projective, then
\begin{equation}\label{eq:3-35}
 M\otimes_R^{\mathbf L}N\simeq M\otimes_RN\in\mathrm D^{\ge d}.
\end{equation}
\end{lemma}

\begin{proof}
For the module resolution and its path-length proof, we follow
\resultcite{\S1, ``The standard resolution'' and ``Consequences'', pp.~7--8}{CB92}.
The derived bounds will then be obtained by truncation and
totalization. It suffices to treat $B=\mathbb C Q$. For $V=(V_i,V_a)$ and
$P_i=Be_i$, consider
\begin{equation}\label{eq:3-36}
 0\longrightarrow\bigoplus_{a:i\to j}P_j\otimes_{\mathbb C}V_i
   \xrightarrow{d}\bigoplus_iP_i\otimes_{\mathbb C}V_i
   \xrightarrow{\varepsilon}V\longrightarrow0,
\end{equation}
where
\[
 \varepsilon(be_i\otimes v)=bv,\qquad
 d(be_j\otimes v)=ba\otimes v-b\otimes V_a(v),\qquad
 \varepsilon d=0.
\]
At vertex $k$, every nonstationary path $p:i\to k$ has a unique
first arrow: $p=p'a$, $a:i\to j$. Hence
\[
 p\otimes v=p'a\otimes v\equiv p'\otimes V_a(v)
       \equiv e_k\otimes V_p(v)\pmod{\operatorname{im}d}.
\]
Consequently,
\[
 e_k(\operatorname{coker}d)\xrightarrow[\varepsilon]{\ \sim\ }V_k,
 \qquad\ker\varepsilon=\operatorname{im}d.
\]
To prove injectivity, group a nonzero source element by paths:
\[
 z=\sum_{a,b}b\otimes v_{a,b},\qquad
 \ell=\max\{|b|:v_{a,b}\ne0\}.
\]
The path-length-$(\ell+1)$ part of $d(z)$ is
\[
 \sum_{|b|=\ell}ba\otimes v_{a,b}\ne0,
\]
because $ba$ uniquely determines $(b,a)$. Thus $\ker d=0$. All
$P_i$ are projective, proving the asserted resolution.

Choose any cochain complex $(C^\bullet,d_C)$ of right $R$-modules
representing $M$. The hypothesis $M\in\mathrm D^{\ge d}$ says
\[
 H^u(C^\bullet)=0\qquad(u<d),
\]
and we use the canonical truncation
$\widetilde M^\bullet=\tau_{\ge d}C^\bullet$ of
\resultcite{Tag 0118}{Stacks}, defined by
\begin{equation}\label{eq:3-37}
 \widetilde M^u=
 \begin{cases}
  0,&u<d,\\
  C^d/\operatorname{im}(d_C^{d-1}),&u=d,\\
  C^u,&u>d.
 \end{cases}
\end{equation}
Its differential is zero below $d$, is induced by $d_C^d$ in degree
$d$, and agrees with $d_C^u$ above $d$:
\[
 d_{\widetilde M}^d([c])=d_C^d(c),\qquad
 d_{\widetilde M}^u=d_C^u\quad(u>d).
\]
The first formula is well-defined because $d_C^dd_C^{d-1}=0$.
The quotient in degree $d$, together with the identity above $d$
and the zero maps below $d$, defines a chain map
\[
 q:C^\bullet\longrightarrow\widetilde M^\bullet.
\]
At the boundary degree,
\[
 H^d(\widetilde M^\bullet)
  =\ker\!\left(C^d/\operatorname{im}d_C^{d-1}\longrightarrow C^{d+1}\right)
  =\frac{\ker d_C^d}{\operatorname{im}d_C^{d-1}}
  =H^d(C^\bullet).
\]
In all other degrees,
\[
 H^u(\widetilde M^\bullet)=
 \begin{cases}
  0=H^u(C^\bullet),&u<d,\\
  H^u(C^\bullet),&u>d.
 \end{cases}
\]
Thus $q$ is a quasi-isomorphism and $\widetilde M^\bullet$
represents $M$.

Choose the projective resolution of $N$ just proved, and denote it by
\[
 P^\bullet=[P^{-1}\xrightarrow{d_P}P^0]\longrightarrow N[0],
 \qquad P^v=0\quad(v\notin\{-1,0\}).
\]
The bounded complex $P^\bullet$ consists of flat left $R$-modules
and is therefore $K$-flat: $\operatorname{Tot}(Z^\bullet\otimes_RP^\bullet)$
is acyclic for every acyclic complex $Z^\bullet$ of right modules
\resultcite{Tag 064K}{Stacks}. In this two-term case, filter that
total complex by its two columns. Its successive quotients are
isomorphic to shifts of the acyclic complexes
$Z^\bullet\otimes_RP^v$, $v=-1,0$. Therefore tensoring with
$P^\bullet$ preserves quasi-isomorphisms and
\begin{equation}\label{eq:3-38}
 M\otimes_R^{\mathbf L}N
      \simeq\operatorname{Tot}(\widetilde M^\bullet\otimes_RP^\bullet).
\end{equation}
Here the total differential on a homogeneous tensor is
\[
 d(m\otimes p)=d_{\widetilde M}(m)\otimes p
        +(-1)^u m\otimes d_P(p),\qquad m\in\widetilde M^u.
\]
The degree-$u$ term is
\begin{equation}\label{eq:3-39}
 \operatorname{Tot}(\widetilde M^\bullet\otimes_RP^\bullet)^u
  =(\widetilde M^u\otimes_RP^0)
       \oplus(\widetilde M^{u+1}\otimes_RP^{-1})
  =0\qquad(u<d-1).
\end{equation}
This proves \eqref{eq:3-34}. If $N$ is projective, take $P^{-1}=0$
and $P^0=N$; then the total complex is zero below $d$, proving
\eqref{eq:3-35}.

The same argument retains auxiliary bimodule actions. Truncation
is performed in the abelian category of bimodules; the resolution
\eqref{eq:3-36} is functorial, with any additional right action on
the $V_i$-factors. The lower bound is checked after forgetting that
additional action.
\end{proof}

\begin{lemma}[{\resultcite{proof of Theorem~2.5}{CLMS19}}]\label{lem:3-12}
Let $R$ be as in Lemma~\ref{lem:3-11}, and let $W$ be its
$S_0$-bimodule of arrows. The sequence
\begin{equation}\label{eq:3-40}
 0\longrightarrow R\otimes_{S_0}W\otimes_{S_0}R
    \xrightarrow{\delta}R\otimes_{S_0}R
    \xrightarrow{\mu}R\longrightarrow0,
\end{equation}
with
\[
 \delta(x\otimes a\otimes y)=xa\otimes y-x\otimes ay,
 \qquad\mu(x\otimes y)=xy,
\]
is a projective bimodule resolution. In particular,
\[
 P\in\mathrm D^{\ge d}\quad\Longrightarrow\quad
 P\otimes_{R^e}^{\mathbf L}R\in\mathrm D^{\ge d-1}.
\]
\end{lemma}

\begin{proof}
The resolution is recalled from
\resultcite{proof of Theorem~2.5, equation~(2.1)}{CLMS19}, for
$R=T_{S_0}(W)$. Semisimplicity of $S_0$ makes its relative
projective terms projective bimodules. The proof below verifies
exactness in a path basis by a longest-path argument and telescoping;
the cited proof gives a contracting homotopy for general tensor
algebras. Treat one factor $B=\mathbb C Q$, with
$S=\bigoplus_i\mathbb C e_i$. The two terms decompose as
\[
 B\otimes_SB\simeq\bigoplus_iBe_i\otimes_{\mathbb C}e_iB,
 \qquad
 B\otimes_SW\otimes_SB
      \simeq\bigoplus_{a:i\to j}Be_j\otimes_{\mathbb C}e_iB.
\]
They are projective bimodules, and $\mu\delta=0$. For a nonzero
linear combination of path tensors $p\otimes a\otimes q$, choose
maximal $|q|=\ell$. The second-path-length-$(\ell+1)$ part of its
image is a nonzero combination of
\[
 -p\otimes aq,\qquad |q|=\ell,
\]
since the last arrow of $aq$ uniquely recovers $a$ and $q$.
Thus $\delta$ is injective.

For $q=a_m\cdots a_1$, telescoping gives
\[
 \sum_{\nu=1}^m
 \delta(pa_m\cdots a_{\nu+1}\otimes a_\nu
                          \otimes a_{\nu-1}\cdots a_1)
      =pq\otimes e_{s(q)}-p\otimes q,
\]
where $s(q)$ is the source of $q$ and empty products are stationary
paths. Therefore
\[
 p\otimes q\equiv pq\otimes e_{s(q)}\pmod{\operatorname{im}\delta},
 \qquad
 (B\otimes_SB)/\operatorname{im}\delta
           \xrightarrow[\mu]{\ \sim\ }B.
\]
This proves exactness. Taking the finite direct sum over the factors
of $R$ gives \eqref{eq:3-40}. Its degrees are $-1,0$, so the final
bound follows as in Lemma~\ref{lem:3-11}.
\end{proof}

\begin{lemma}\label{lem:3-13}
Suppose $R$ is a product of path algebras of finite acyclic quivers
and $A^1$ is projective as a left $R$-module. In the notation of
Lemma~\ref{lem:3-10}, set $s=\#\{i:q_i=2\}$. Then
\[
 K_{t,\mathbf q}\in\mathrm D^{\ge-s-1},\qquad
 G_{t,\mathbf q}\in\mathrm D^{\ge t-1}.
\]
\end{lemma}

\begin{proof}
Regard the selected internal-degree pieces as modules in degree
zero, and put
\[
 H_0=T^t,\qquad H_i=H_{i-1}\otimes_R^{\mathbf L}A^{q_i},
 \qquad s_i=\#\{j\le i:q_j=2\}.
\]
Lemma~\ref{lem:3-11} gives
\[
 H_{i-1}\in\mathrm D^{\ge-s_{i-1}}
 \quad\Longrightarrow\quad
 H_i\in
 \begin{cases}
  \mathrm D^{\ge-s_{i-1}},&q_i=1,\\
  \mathrm D^{\ge-s_{i-1}-1},&q_i=2,
 \end{cases}
 \ =\mathrm D^{\ge-s_i}.
\]
Hence $H_r\in\mathrm D^{\ge-s}$, and Lemma~\ref{lem:3-12} gives
$K_{t,\mathbf q}\in\mathrm D^{\ge-s-1}$. By
Lemma~\ref{lem:3-10},
\[
 G_{t,\mathbf q}\simeq K_{t,\mathbf q}
                      \left[r-t-\sum_iq_i\right],
\]
and
\[
 -s-1+t+\sum_iq_i-r=-s-1+t+(r+s)-r=t-1.
\]
Thus $G_{t,\mathbf q}\in\mathrm D^{\ge t-1}$, also for $r=s=0$.
\end{proof}

Convergence for finite filtrations implies the next vanishing criterion
\resultcite{Lemma~12.24.11, Tag 012W}{Stacks}.

\begin{lemma}\label{lem:3-14}
Let $0=F_{-1}C\subseteq F_0C\subseteq\cdots\subseteq F_mC=C$
be a finite filtration by subcomplexes. If
$\operatorname{gr}_r^F C\in\mathrm D^{\ge d}$ for every $r$, then
$C\in\mathrm D^{\ge d}$.
\end{lemma}

\begin{proof}
The short exact sequences
$0\to F_{r-1}C\to F_rC\to\operatorname{gr}_r^F C\to0$ give,
for $u<d$,
\[
 H^u(F_{r-1}C)=0\quad\Longrightarrow\quad
 0\longrightarrow H^u(F_rC)\longrightarrow
       H^u(\operatorname{gr}_r^F C)=0.
\]
Induct on $r$, starting from $F_{-1}C=0$.
\end{proof}

\begin{lemma}\label{lem:3-15}
Let $A=A^0\oplus A^1\oplus A^2$ be a finite graded directed
algebra, with $S_0,J,R,J_0$ as in
\eqref{eq:3-6}--\eqref{eq:3-7}. Suppose $R$ is a product of path
algebras of finite acyclic quivers and $A^1$ is projective as a
left $R$-module. If $T$ is a graded $A$-bimodule with $T^t=0$
for $t<t_0$, then
\[
 H^u C(A,T)=0\qquad(u<t_0-1).
\]
\end{lemma}

\begin{proof}
Lemmas~\ref{lem:3-9} and~\ref{lem:3-13} give
\[
 \operatorname{gr}_r^F C(A,T)
 =\bigoplus_{t\ge t_0}\ \bigoplus_{\mathbf q\in\{1,2\}^r}
       G_{t,\mathbf q}\in\mathrm D^{\ge t_0-1}.
\]
The filtration is finite by Lemma~\ref{lem:3-9};
Lemma~\ref{lem:3-14} yields $C(A,T)\in\mathrm D^{\ge t_0-1}$.
\end{proof}

We now apply the algebraic estimate to the normal bimodule.
Proposition~\ref{prop:3-8} supplies the hereditary degree-zero
algebra and the projectivity of $A^1$ from the geometric condition
\eqref{eq:3-3}. The inverse Serre functor places the normal bimodule
in internal degrees at least two. Together, these two facts turn
\eqref{eq:3-3} into a positive-height obstruction to fullness.

\begin{proposition}\label{prop:3-16}
Let $\mathcal E$ be a nonempty exceptional collection of line
bundles on a smooth projective complex surface. If
\eqref{eq:3-3} holds, then $h_X(\mathcal E)\ge1$. In particular,
$\mathcal E$ is not full. This applies to every such collection
when $K_X$ is nef.
\end{proposition}

\begin{proof}
Exceptionality gives $p_g=q=0$ and
$A=H^*(\mathcal A)=A^0\oplus A^1\oplus A^2$.
For $T=H^*(\mathcal T)$, \eqref{eq:2-1} gives
\[
 e_iT^te_j=\Ext^t(\cL_j,S_X^{-1}\cL_i)
   =H^{t-2}(X,\cL_j^\vee\otimes\cL_i\otimes\omega_X^{-1})=0
       \qquad(t<2).
\]
Proposition~\ref{prop:3-8} and Lemma~\ref{lem:3-15}, with $t_0=2$,
imply in \eqref{eq:2-7}
\[
\begin{aligned}
 E_2^{-p,q}=0\quad(q-p<1)
 &\quad\Longrightarrow\quad E_\infty^{-p,q}=0\quad(q-p<1)\\
 &\quad\Longrightarrow\quad\NHH^u(\mathcal E,X)=0\quad(u<1).
\end{aligned}
\]
Thus $h_X(\mathcal E)\ge1$, and Proposition~\ref{prop:2-2} excludes
fullness. Nefness of $K_X$ implies \eqref{eq:3-3}.
\end{proof}

The $E_2$-vanishing uses Yoneda multiplication through
Proposition~\ref{prop:3-8} and Lemma~\ref{lem:3-15}.
For a full collection, the contrapositive forces a prime component
of an effective forward difference to have negative canonical
degree. After moving the blowup centers as below, the classification
of such differences will identify an exceptional curve and provide
the pair of line bundles needed for descent.

\subsection{Reduction to blowups at general centers}\label{subsec:general-centers}

To apply the height obstruction to a nonminimal surface, we first move the blowup centers away from the rational curves of its minimal model.  Three facts are needed.  Ordered towers of point blowups, including infinitely near centers, occur in one smooth irreducible family, and every line bundle on a fiber extends to that family.  Fullness persists on an open neighborhood of a fiber carrying a full exceptional collection.  Finally, on a non-uniruled surface the rational curves form a countable family, so a configuration avoiding all of them can be chosen inside any prescribed nonempty open set.  We establish these facts separately and combine them in Corollary~\ref{cor:3-23}.

We begin with openness of fullness.  Proposition~\ref{prop:3-19} is the line-bundle form of \resultcite{Proposition~3.9}{Hu19}; its proof also describes the projection kernel that detects fullness.

\begin{lemma}[\resultcite{Theorem~1.2}{CS11}]\label{lem:3-17}
Let $X$ be smooth and projective, and let $K\in\Perf(X\times X)$.  For the integral transform
\[
 \Phi_K(F)=Rp_{2*}(Lp_1^*F\otimes^{\mathbf L}K),
\]
one has $\Phi_K=0$ if and only if $K=0$.
\end{lemma}

\begin{proof}
This is the zero-functor case of \resultcite{Theorem~1.2}{CS11}.  We recall a direct proof using derived fibers.  For a closed point $x\in X$, write $i_x\colon X\hookrightarrow X\times X$, $y\mapsto(x,y)$.  Then
\[
 \Phi_K(k(x))\simeq Li_x^*K.
\]
Consequently,
\[
 \Phi_K=0
 \quad\Longrightarrow\quad
 K\otimes^{\mathbf L}k(x,y)=0\text{ for all closed }(x,y)
 \quad\Longrightarrow\quad K=0,
\]
by derived Nakayama \resultcite{Tag~0BCD}{Stacks}.  The converse follows from the definition of $\Phi_K$.
\end{proof}

\begin{lemma}[cf.\ \resultcite{Lemmas~2.10--2.12}{Hu19}]\label{lem:3-18}
Let $(\cL_1,\ldots,\cL_\gamma)$ be an exceptional collection of line bundles on $X$.  For the diagonal $\Delta\subset X\times X$, define
\[
 C_i=\Cone(\cL_i^\vee\boxtimes\cL_i\longrightarrow\cO_\Delta),
 \qquad P=\Phi_{C_1}\circ\cdots\circ\Phi_{C_\gamma},
\]
so that $\Phi_{C_\gamma}$ is applied first.  The functor $P$ is the projection to $\langle\cL_1,\ldots,\cL_\gamma\rangle^\perp$.  If $C$ is its convolution kernel, then
\[
 \langle\cL_1,\ldots,\cL_\gamma\rangle=\Db(X)
 \quad\Longleftrightarrow\quad P=0
 \quad\Longleftrightarrow\quad C=0.
\]
\end{lemma}

\begin{proof}
The argument below specializes and expands the iterated evaluation construction in \resultcite{proof of Lemma~2.12}{Hu19}, with the mutation shifts fixed by Subsection~\ref{subsec:exceptional-mutations}.  Its kernel form follows from the convolution formula \resultcite{Tag~0FYP, Lemma~57.8.3}{Stacks}.  For $F_\gamma=F$, define successively
\[
 F_{i-1}=\Phi_{C_i}(F_i),
 \qquad
 \RHom(\cL_i,F_i)\otimes\cL_i\longrightarrow F_i\longrightarrow F_{i-1}
 \qquad(i=\gamma,\ldots,1).
\]
Exceptionality and semiorthogonality give
\begin{align*}
 \RHom(\cL_i,F_{i-1})&=0,\\
 \RHom(\cL_j,F_{i-1})&\simeq\RHom(\cL_j,F_i)\qquad(j>i).
\end{align*}
Thus $P(F)=F_0\in\langle\cL_1,\ldots,\cL_\gamma\rangle^\perp$, and
$P|_{\langle\cL_1,\ldots,\cL_\gamma\rangle^\perp}\simeq\id$.
The same triangles imply
\[
 P(F)=0\quad\Longrightarrow\quad F\in\langle\cL_1,\ldots,\cL_\gamma\rangle.
\]
The equivalences follow, using Lemma~\ref{lem:3-17} for the last one.
\end{proof}

\begin{proposition}[\resultcite{Proposition~3.9}{Hu19}]\label{prop:3-19}
Let $f\colon\mathcal X\to B$ be a smooth projective family of connected surfaces over a smooth complex variety, and let $\widetilde{\cL}_1,\ldots,\widetilde{\cL}_\gamma$ be line bundles on $\mathcal X$.  If their restrictions form a full exceptional collection on $X_{b_0}$, they do so over an open neighborhood of $b_0$.
\end{proposition}

\begin{proof}
To apply \resultcite{Proposition~3.9}{Hu19}, we first establish relative exceptionality for the given extensions.  For a point $b\in B$, let
\[
 i_b\colon\operatorname{Spec}k(b)\longrightarrow B,
 \qquad X_b=\mathcal X\times_B\operatorname{Spec}k(b),
 \qquad \cL_{i,b}=\widetilde{\cL}_i|_{X_b}.
\]
Define the unit morphism $u\colon\cO_B\to Rf_*\cO_{\mathcal X}$ and the complexes
\begin{equation}\label{eq:3-41}
 K_0:=\Cone(u),
 \qquad
 K_{ji}:=Rf_*(\widetilde{\cL}_j^\vee\otimes\widetilde{\cL}_i)
 \qquad(1\le i<j\le \gamma).
\end{equation}
These complexes are perfect and satisfy arbitrary derived base change \resultcite{Tag~0DJT}{Stacks}.

The role of $K_0$ is to control the diagonal endomorphism complexes.  Since $\widetilde{\cL}_i$ is invertible,
\[
 R\mathcal Hom_{\mathcal X}(\widetilde{\cL}_i,\widetilde{\cL}_i)
 \simeq\widetilde{\cL}_i^\vee\otimes\widetilde{\cL}_i
 \simeq\cO_{\mathcal X}.
\]
Under this identification, $u$ is the relative identity morphism for every $\widetilde{\cL}_i$.  On a fiber, write $K_{0,b}=Li_b^*K_0$.  Derived base change gives the distinguished triangle
\begin{equation}\label{eq:3-42}
 k(b)\xrightarrow{\,1\mapsto\id_{\cL_{i,b}}\,}
 \RHom_{X_b}(\cL_{i,b},\cL_{i,b})
 \simeq R\Gamma(X_b,\cO_{X_b})
 \longrightarrow K_{0,b}\longrightarrow k(b)[1].
\end{equation}
Consequently,
\begin{equation}\label{eq:3-43}
 \begin{aligned}
 K_{0,b}=0
 &\quad\Longleftrightarrow\quad
 k(b)\xrightarrow{\sim}\RHom_{X_b}(\cL_{i,b},\cL_{i,b})\\
 &\quad\Longleftrightarrow\quad
 \cL_{i,b}\text{ is exceptional for every }i.
 \end{aligned}
\end{equation}
For a closed complex point $b$, connectedness and projectivity imply
$H^0(X_b,\cO_{X_b})=\mathbb C$, and the first arrow induces the identity on $H^0$.  Since $\dim X_b=2$, the cohomology sequence gives
\[
 H^v(K_{0,b})=
 \begin{cases}
 H^v(X_b,\cO_{X_b}),&v=1,2,\\
 0,&v\notin\{1,2\}.
 \end{cases}
\]
Thus $K_{0,b}=0$ is precisely the condition $q(X_b)=p_g(X_b)=0$.  This condition is independent of $i$.

The other complexes control reverse semiorthogonality:
\begin{equation}\label{eq:3-44}
 Li_b^*K_{ji}
 \simeq R\Gamma(X_b,\cL_{j,b}^\vee\otimes\cL_{i,b})
 \simeq\RHom_{X_b}(\cL_{j,b},\cL_{i,b})
 \qquad(j>i).
\end{equation}
The hypothesis at $b_0$ gives
\[
 Li_{b_0}^*K_0=0,
 \qquad Li_{b_0}^*K_{ji}=0\qquad(j>i).
\]
For a perfect complex $K$ on $B$, write
$\Supp K=\bigcup_v\Supp H^v(K)$.  This is closed, and derived Nakayama gives
\[
 b\notin\Supp K\quad\Longleftrightarrow\quad Li_b^*K=0;
\]
see \resultcite{Tag~0BCD}{Stacks}.  Hence
\begin{equation}\label{eq:3-45}
 U=B\setminus\left(\Supp K_0\cup\bigcup_{j>i}\Supp K_{ji}\right)
\end{equation}
is an open neighborhood of $b_0$.  On $\mathcal X_U\to U$ we have
\[
 \cO_U\xrightarrow{\sim}Rf_{U*}\cO_{\mathcal X_U},
 \qquad
 Rf_{U*}(\widetilde{\cL}_j^\vee\otimes\widetilde{\cL}_i)=0
 \qquad(j>i).
\]
These are precisely relative exceptionality for the restricted line bundles
\resultcite{Definition~2.6 and Lemma~2.7}{Hu19}.  Since they are full on $X_{b_0}$,
\resultcite{Proposition~3.9}{Hu19} now gives an open neighborhood $U'\subseteq U$ on which the same collection is full.

The open neighborhood $U'$ can be constructed from the projection kernel of Lemma~\ref{lem:3-18}; the proof in \resultcite{Theorem~3.5 and Proposition~3.9}{Hu19} uses helices.  Let $p_1,p_2$ and $\Delta$ be the projections and diagonal of $\mathcal X_U\times_U\mathcal X_U$.  With the line bundles restricted to $\mathcal X_U$, put
\[
 C_i=\Cone\left(p_1^*\widetilde{\cL}_i^\vee\otimes p_2^*\widetilde{\cL}_i
 \longrightarrow\cO_\Delta\right).
\]
Let $C$ be the convolution kernel of
$\Phi_{C_1}\circ\cdots\circ\Phi_{C_\gamma}$, with $\Phi_{C_\gamma}$ applied first.  The diagonal is regular, and convolution is a derived pullback--tensor--pushforward construction
\resultcite{Tag~0FYP, Lemma~57.8.3}{Stacks}.  Proper smooth pushforward and derived base change \resultcite{Tag~0DJT}{Stacks} give
\[
 C\in\Perf(\mathcal X_U\times_U\mathcal X_U),
 \qquad
 \Phi_{C_b}\simeq\Phi_{(C_1)_b}\circ\cdots\circ\Phi_{(C_\gamma)_b}.
\]
Here $C_b$ is the derived restriction to $X_b\times X_b$.
Lemma~\ref{lem:3-18} and fullness at $b_0$ imply $C_{b_0}=0$.  For the proper structure morphism $q\colon\mathcal X_U\times_U\mathcal X_U\to U$, take
\[
 U'=U\setminus q(\Supp C).
\]
Properness makes $U'$ open.  Derived Nakayama applied along the fiber over $b_0$ gives $b_0\in U'$.  For every $b\in U'$,
\[
 C_b=0\quad\Longrightarrow\quad
 \Db(X_b)=\langle\cL_{1,b},\ldots,\cL_{\gamma,b}\rangle,
\]
by Lemma~\ref{lem:3-18}.
\end{proof}

\begin{lemma}[cf.\ \resultcite{Propositions~1.4--1.5 and \S2}{Roe01}]\label{lem:3-20}
Fix a smooth projective connected surface $S$ and $r\ge0$.  There exist a smooth irreducible projective variety $B_r$, a smooth projective family $f_r\colon\mathcal X_r\to B_r$, and $g_r\colon\mathcal X_r\to S$ such that:
\begin{enumerate}[label=\textup{(\roman*)}]
 \item Every ordered tower of $r$ point blowups over $S$, including infinitely near centers, occurs as a fiber.
 \item The locus of distinct images of the centers is the nonempty open set
 \[
 \Conf_r(S)=S^r\setminus\bigcup_{a<b}\{p_a=p_b\}.
 \]
 \item Relative line bundles restrict to the total transforms of all successive exceptional divisors.
 \item The restriction $\Pic(\mathcal X_r)\to\Pic((\mathcal X_r)_b)$ is surjective for every closed point $b\in B_r(\mathbb C)$.
\end{enumerate}
For $r=0$, one has $B_0=\operatorname{Spec}\mathbb C$ and $\mathcal X_0=S$.
\end{lemma}

\begin{proof}
We recall the ordered-cluster construction and its fiber identifications from \resultcite{Propositions~1.4--1.5 and their proofs}{Roe01}, in the notation needed here.  The relative exceptional divisors are those constructed in \resultcite{\S2, before Remark~2.1}{Roe01}.  Part~(iv) follows from this construction and the Picard decomposition in \resultcite{\S1, before Lemma~1.2}{Roe01}.

Starting with $(B_0,\mathcal X_0)=(\operatorname{Spec}\mathbb C,S)$, define
\begin{align*}
 B_r&=\mathcal X_{r-1},
 &\widetilde{\mathcal X}_{r-1}&=\mathcal X_{r-1}\times_{B_{r-1}}B_r,\\
 \sigma_r&\colon B_r\longrightarrow\widetilde{\mathcal X}_{r-1},
 &\sigma_r(b)&=(b,b),\\
 \beta_r&\colon\mathcal X_r=\operatorname{Bl}_{\sigma_r(B_r)}\widetilde{\mathcal X}_{r-1}
 \longrightarrow\widetilde{\mathcal X}_{r-1},\\
 f_r&=\operatorname{pr}_2\circ\beta_r,
 &g_r&=g_{r-1}\circ\operatorname{pr}_1\circ\beta_r.
\end{align*}
The section has relative codimension two.  In relative local coordinates,
$\operatorname{Bl}_0\mathbb A^2_{B_r}$ has smooth charts
\[
 (u,w)\longmapsto(u,uw),
 \qquad(z,v)\longmapsto(zv,v).
\]
Thus $f_r$ is smooth and projective, and formation commutes with fibers:
\[
 (\mathcal X_r)_b\simeq
 \operatorname{Bl}_b\bigl((\mathcal X_{r-1})_{f_{r-1}(b)}\bigr).
\]
Induction gives smoothness and irreducibility of $B_r,\mathcal X_r$ and proves~(i).

The images of the centers define a morphism $\eta_r\colon B_r\to S^r$.  Put $\Delta_{ab}=\{p_a=p_b\}$.  Each later center has a unique lift when its image avoids the previous ones, so
\[
 \eta_r^{-1}\left(S^r\setminus\bigcup_{a<b}\Delta_{ab}\right)
 \xrightarrow{\sim}\Conf_r(S).
\]
This proves~(ii).

Let $\mathcal E_r$ be the last relative exceptional divisor.  Define
\[
 \mathcal H_{r,r}=\cO_{\mathcal X_r}(\mathcal E_r),
 \qquad
 \mathcal H_{a,r}=\beta_r^*\operatorname{pr}_1^*\mathcal H_{a,r-1}
 \qquad(a<r).
\]
Then $\mathcal H_{a,r}|_{(\mathcal X_r)_b}\simeq\cO(E_a^{\mathrm{tot}})$, proving~(iii).  On a closed fiber $\pi\colon X\to S$, the Picard group decomposition gives
\begin{equation}\label{eq:3-46}
 \cL\simeq\pi^*\cM\otimes\cO_X\left(\sum_{a=1}^r m_aE_a^{\mathrm{tot}}\right),
 \qquad\cM\in\Pic(S),\quad m_a\in\mathbb Z.
\end{equation}
The extension
\[
 \widetilde{\cL}=g_r^*\cM\otimes\bigotimes_{a=1}^r\mathcal H_{a,r}^{\otimes m_a},
 \qquad \widetilde{\cL}|_X\simeq\cL,
\]
proves~(iv).
\end{proof}

\begin{lemma}[\resultcite{27(1)}{Kol10}]\label{lem:3-21}
A non-uniruled smooth projective complex surface $S$ contains at most countably many integral rational curves.
\end{lemma}

\begin{proof}
We give the surface specialization of the parameter-space argument in \resultcite{Theorem~25 and 27(1)}{Kol10}, including the reduction of a dominant evaluation map to a family over a curve.  Fix a projective embedding and set
\[
 \mathcal H_d=\operatorname{Hom}_d(\mathbf P^1,S),\qquad d\ge1.
\]
Each $\mathcal H_d$ is of finite type: it is defined by the equations of $S$ in the open parameter space of tuples of degree-$d$ homogeneous polynomials without a common zero.  Hence
\[
 \mathcal T=\bigcup_{d\ge1}\operatorname{Irr}(\mathcal H_d)
\]
is countable.

For $T\in\mathcal T$, with its reduced structure, consider
\[
 \operatorname{ev}_T\colon\mathbf P^1\times T\longrightarrow S,
 \qquad Z_T=\overline{\operatorname{im}(\operatorname{ev}_T)}.
\]
If $Z_T=S$, generic smoothness supplies $(z,t)\in\mathbf P^1\times T_{\mathrm{sm}}$ with
\[
 \operatorname{rank}d\operatorname{ev}_{T,(z,t)}=2,
 \qquad
 d\operatorname{ev}_{T,(z,t)}|_{T_z\mathbf P^1}\ne0.
\]
Choose a curve $D\subset T$ through $t$ whose tangent direction complements the image of $T_z\mathbf P^1$.  Then
\[
 \operatorname{rank}d(\operatorname{ev}_T|_{\mathbf P^1\times D})_{(z,t)}=2,
\]
so $\mathbf P^1\times D\to S$ is dominant.  Completing and normalizing $D$ gives a dominant rational map $\mathbf P^1\times\overline D\dashrightarrow S$, contradicting non-uniruledness.  Therefore $\dim Z_T=1$; all maps in $T$ have the same integral rational image $Z_T$.  Every integral rational curve is the image of its normalization, so the countable family $(Z_T)_{T\in\mathcal T}$ contains all of them.
\end{proof}

\begin{lemma}[cf.\ \resultcite{Tag~0CQN}{Stacks}]\label{lem:3-22}
Let $S$ be as in Lemma~\ref{lem:3-21}.  For $r>0$, every nonempty Zariski-open $U\subset\Conf_r(S)$ contains a tuple $(p_1,\ldots,p_r)$ such that no $p_a$ lies on an integral rational curve of $S$.
\end{lemma}

\begin{proof}
Enumerate the rational curves as $C_n$, with the finite or empty cases included.  For each coordinate projection $\operatorname{pr}_a\colon S^r\to S$, the subset
\[
 Z_{n,a}=U\cap\operatorname{pr}_a^{-1}(C_n)
\]
is proper and closed in $U$.  Choose a nonempty analytic open ball $\Omega\subset U$.  Each $Z_{n,a}\cap\Omega$ is closed and nowhere dense in $\Omega$.  Applying the Baire category theorem \resultcite{Tag~0CQN}{Stacks} to the locally compact Hausdorff space $\Omega$ gives
\[
 \Omega\setminus\bigcup_{n,a}Z_{n,a}\ne\varnothing.
\]
Every point of this complement has the required property.
\end{proof}

\begin{corollary}\label{cor:3-23}
Let $S$ be a non-uniruled smooth connected projective complex surface, and let $X\to S$ be a tower of $r\ge0$ point blowups.  If $X$ admits a full exceptional collection of line bundles, then there exist pairwise distinct points $p_1,\ldots,p_r\in S$, none lying on any integral rational curve of $S$, such that
$Y=\operatorname{Bl}_{p_1,\ldots,p_r}S$ admits a full exceptional collection of line bundles.  For $r=0$, take $Y=S$.
\end{corollary}

\begin{proof}
The assertion for $r=0$ is immediate.  Suppose $r>0$, and denote the given collection by
$\cE^0=(\cL_1^0,\ldots,\cL_\gamma^0)$.  By Lemma~\ref{lem:3-20}, there are a smooth projective family $f_r\colon\mathcal X_r\to B_r$ with $B_r$ irreducible, a point $b_0\in B_r$, and line bundles $\widetilde{\cL}_i$ on $\mathcal X_r$ satisfying
\[
 (\mathcal X_r)_{b_0}\simeq X,
 \qquad\widetilde{\cL}_i|_{(\mathcal X_r)_{b_0}}\simeq\cL_i^0
 \qquad(1\le i\le \gamma).
\]
For $b\in B_r(\mathbb C)$, put
\[
 X_b=(\mathcal X_r)_b,
 \qquad\cL_{i,b}=\widetilde{\cL}_i|_{X_b},
 \qquad\cE_b=(\cL_{1,b},\ldots,\cL_{\gamma,b}).
\]
Proposition~\ref{prop:3-19} gives a Zariski-open neighborhood $V\subseteq B_r$ of $b_0$ such that $\cE_b$ is full and exceptional for every $b\in V(\mathbb C)$.

The distinct-center locus of Lemma~\ref{lem:3-20}(ii) is the image of an open immersion
\[
 j_r\colon\Conf_r(S)\hookrightarrow B_r,
 \qquad\Omega_r:=j_r(\Conf_r(S)),
 \qquad X_{j_r(p)}\simeq\operatorname{Bl}_{p_1,\ldots,p_r}S,
\]
where $p=(p_1,\ldots,p_r)\in\Conf_r(S)(\mathbb C)$.  Define
\[
 W:=j_r^{-1}(V)\subseteq\Conf_r(S).
\]
Both $V$ and $\Omega_r$ are nonempty open subsets of the irreducible variety $B_r$.  Hence
\[
 \varnothing\ne V\cap\Omega_r=j_r(W),
\]
so $W$ is a nonempty Zariski-open subset of $\Conf_r(S)$.  For every $p\in W(\mathbb C)$, the collection $\cE_{j_r(p)}$ is therefore full and exceptional on $\operatorname{Bl}_{p_1,\ldots,p_r}S$.

It remains to choose a point of $W$ satisfying the avoidance condition.  Let $\mathcal R(S)$ be the set of integral rational curves on $S$, which is at most countable by Lemma~\ref{lem:3-21}.  For $1\le a\le r$, let $\operatorname{pr}_a\colon\Conf_r(S)\to S$ be the coordinate projection, and define
\[
 Z_{C,a}:=W\cap\operatorname{pr}_a^{-1}(C)
 \qquad(C\in\mathcal R(S)).
\]
These are proper closed subsets of $W$: each is the intersection of the nonempty open subset $W\subset S^r$ with the proper closed subset $\{p\in S^r:p_a\in C\}$.  Apply Lemma~\ref{lem:3-22} with its open set $U$ equal to $W$.  It gives
\[
 p\in W(\mathbb C)\setminus
 \bigcup_{C\in\mathcal R(S)}\bigcup_{a=1}^rZ_{C,a}(\mathbb C).
\]
For this chosen tuple, put $b=j_r(p)$ and $Y=X_b\simeq\operatorname{Bl}_{p_1,\ldots,p_r}S$.  Then
\begin{align*}
 p\in\Conf_r(S)&\quad\Longrightarrow\quad p_a\ne p_{a'}\qquad(a\ne a'),\\
 p\notin\bigcup_{C,a}Z_{C,a}
 &\quad\Longrightarrow\quad p_a\notin C
 \qquad(1\le a\le r,\ C\in\mathcal R(S)),\\
 p\in W(\mathbb C)&\quad\Longrightarrow\quad b\in V(\mathbb C).
\end{align*}
The last inclusion makes $\cE_b$ a full exceptional collection of line bundles on $Y$, as required.
\end{proof}

\subsection{Finding an exceptional curve in the collection}\label{subsec:exceptional-pair}

Let $p_g(S)=q(S)=0$, let $K_S$ be nef, and let
\[
 \pi\colon Y=\operatorname{Bl}_{p_1,\ldots,p_r}S\longrightarrow S
\]
be as in Corollary~\ref{cor:3-23}.  The exceptional curves satisfy
\begin{equation}\label{eq:3-47}
 E_a^2=-1,\qquad E_aE_b=0\quad(a\ne b),
 \qquad K_Y=\pi^*K_S+\sum_aE_a.
\end{equation}
The choice of centers separates every rational curve on $Y$ from the exceptional locus unless it is itself exceptional.  Applied to the zero divisors of forward morphisms, this observation has two consequences.  Proposition~\ref{prop:3-16} forces a pair of line bundles differing by an exceptional curve, and the cohomology of that curve then produces such a pair in adjacent positions.  The latter step is needed to perform the mutations used in the contraction argument.

\begin{lemma}\label{lem:3-24}
Every effective left-orthogonal divisor $D>0$ on $Y$ is either a single exceptional curve $E_a$ or the pullback of an effective left-orthogonal divisor $D_0>0$ on $S$.  In the latter case,
\[
 \Supp D\cap\bigcup_aE_a=\varnothing,
 \qquad\chi(\cO_Y(D))=-K_SD_0\le0.
\]
\end{lemma}

\begin{proof}
By Lemma~\ref{lem:3-1}, $\Supp D$ is connected and its prime components are smooth rational curves.  For a noncontracted component $C$, its image $C_0=\pi(C)$ is rational and avoids all $p_a$.  Thus
\[
 C=\pi^*C_0,
 \qquad C\cap\bigcup_aE_a=\varnothing.
\]
Connectedness yields the alternatives
\[
 D=mE_a\quad(m\ge1),
 \qquad\text{or}\qquad
 D=\pi^*D_0,\quad\Supp D\cap\bigcup_aE_a=\varnothing.
\]
In the first case, \eqref{eq:2-3} gives
\[
 -m^2-m=D^2+K_YD=-2\quad\Longrightarrow\quad m=1.
\]
In the second, $R\pi_*\cO_Y\simeq\cO_S$ and the projection formula give
\[
 R\Gamma(S,\cO_S(-D_0))\simeq R\Gamma(Y,\cO_Y(-D))=0,
\]
so $D_0$ is left-orthogonal.  Finally,
\[
 \chi(\cO_Y(D))=\chi(\cO_S(D_0))=-K_SD_0\le0.
 \qedhere
\]
\end{proof}

\begin{corollary}\label{cor:3-25}
If $(\cL,\cM)$ is an exceptional pair of line bundles on $Y$ and $\RHom(\cL,\cM)$ is nonzero and concentrated in degree zero, then
\[
 \cM\simeq\cL(E_a)\text{ for some }a,
 \qquad\RHom(\cL,\cM)\simeq\mathbb C.
\]
Conversely, every pair $(\cL,\cL(E_a))$ is exceptional with this forward morphism complex.
\end{corollary}

\begin{proof}
A nonzero morphism has an effective left-orthogonal divisor $D>0$.  The pullback alternative of Lemma~\ref{lem:3-24} would imply
\[
 0<h^0(Y,\cO_Y(D))=\chi(\cO_Y(D))\le0.
\]
Hence $D=E_a$.  Conversely, the sequences
\begin{gather*}
 0\longrightarrow\cO_Y(-E_a)\longrightarrow\cO_Y\longrightarrow\cO_{E_a}\longrightarrow0,\\
 0\longrightarrow\cO_Y\longrightarrow\cO_Y(E_a)\longrightarrow\cO_{E_a}(-1)\longrightarrow0
\end{gather*}
and $R\Gamma(\cO_Y)\simeq\mathbb C$ give
\[
 R\Gamma(\cO_Y(-E_a))=0,
 \qquad R\Gamma(\cO_Y(E_a))\simeq\mathbb C.
\]
Tensoring by $\cL$ proves the assertion.
\end{proof}

\begin{lemma}[cf.\ \resultcite{Tag~0AA4}{Stacks}]\label{lem:3-26}
Let $\iota\colon E\hookrightarrow Y$ be a $(-1)$-curve on a smooth surface, and regard $\cO_E(d)$ as a sheaf on $Y$.  For every line bundle $\cM$,
\begin{align}
 \RHom_Y(\cM,\cO_E(d))
 &\simeq R\Gamma(\mathbf P^1,\cO(d-\cM\cdot E)),\label{eq:3-48}\\
 \RHom_Y(\cO_E(d),\cM)
 &\simeq R\Gamma(\mathbf P^1,\cO(\cM\cdot E-d-1))[-1].\label{eq:3-49}
\end{align}
In particular,
\begin{align*}
 \RHom_Y(\cO_E(d),\cM)=0
 &\quad\Longleftrightarrow\quad\cM\cdot E=d,\\
 \RHom_Y(\cM,\cO_E(d))=0
 &\quad\Longleftrightarrow\quad\cM\cdot E=d+1.
\end{align*}
\end{lemma}

\begin{proof}
We apply the effective-Cartier-divisor duality formula \resultcite{Tag~0AA4}{Stacks} to $\iota$, and compute the resulting line-bundle cohomology on $E\simeq\mathbf P^1$.  One has
\[
 L\iota^*\cM=\cM|_E\simeq\cO_E(\cM\cdot E),
 \qquad
 \iota^!\cM\simeq\cM|_E\otimes\cO_E(E)[-1]
 \simeq\cO_E(\cM\cdot E-1)[-1].
\]
Adjunction gives \eqref{eq:3-48}--\eqref{eq:3-49}.  The vanishings follow from
\[
 R\Gamma(\mathbf P^1,\cO(n))=0\quad\Longleftrightarrow\quad n=-1.
 \qedhere
\]
\end{proof}

\begin{lemma}\label{lem:3-27}
If $r>0$ and $(\cL_1,\ldots,\cL_\gamma)$ is a full exceptional collection of line bundles on $Y$, then $\cL_j\simeq\cL_i(E_a)$ for some $i<j$ and some $a$.
\end{lemma}

\begin{proof}
Otherwise every nonzero forward morphism has a pullback divisor by Lemma~\ref{lem:3-24}.  Each of its prime components $C$ satisfies
\[
 CE_a=0\text{ for all }a,
 \qquad K_YC=K_S\pi_*C\ge0.
\]
Proposition~\ref{prop:3-16} and fullness would then give
\[
 1\le h_Y(\cL_1,\ldots,\cL_\gamma)=0,
\]
contrary to Proposition~\ref{prop:2-2}.
\end{proof}

\begin{proposition}\label{prop:3-28}
If $r>0$ and $(\cL_1,\ldots,\cL_\gamma)$ is a full exceptional collection of line bundles on $Y$, it contains an adjacent pair $(\cL,\cL(E_a))$ for some exceptional curve $E_a$.
\end{proposition}

\begin{proof}
By Lemma~\ref{lem:3-27}, choose $\cL_j\simeq\cL_i(E_a)$ with
\[
 j-i=\min\{v-u:u<v,\ \cL_v\simeq\cL_u(E_b)\text{ for some }b\}.
\]
Set $E=E_a$, $\ell=\cL_i\cdot E$, and $T=\cO_E(\ell-1)$.  Then
\begin{equation}\label{eq:3-50}
 0\longrightarrow\cL_i\longrightarrow\cL_j\longrightarrow T\longrightarrow0.
\end{equation}
For $i<k<j$, put $t=\cL_k\cdot E-\ell$.  Semiorthogonality and Lemma~\ref{lem:3-26} give
\begin{align*}
 \RHom(\cL_i,\cL_k)
 &\simeq\RHom(T,\cL_k)[1]
 \simeq R\Gamma(\mathbf P^1,\cO(t)),\\
 \RHom(\cL_k,\cL_j)
 &\simeq\RHom(\cL_k,T)
 \simeq R\Gamma(\mathbf P^1,\cO(-t-1)).
\end{align*}
Since $t\in\mathbb Z$, precisely one of the following applies:
\[
 \begin{array}{c@{\qquad}c@{\qquad}c}
 &\text{nonzero degree-zero complex}&\text{index gap}\\[2pt]
 t\ge0&\RHom(\cL_i,\cL_k)\simeq\mathbb C^{t+1}&k-i<j-i\\[2pt]
 t\le-1&\RHom(\cL_k,\cL_j)\simeq\mathbb C^{-t}&j-k<j-i.
 \end{array}
\]
Corollary~\ref{cor:3-25} produces an exceptional-curve difference of smaller gap in either case.  Thus no $k$ satisfies $i<k<j$, and $j=i+1$.
\end{proof}

\subsection{Descending through a contraction}\label{subsec:contraction}

Let $f\colon Y\to Z$ contract a $(-1)$-curve $E$, with $p_g(Y)=q(Y)=0$.  Once an adjacent pair $(\cL,\cL(E))$ has been found, its evaluation sequence replaces one line bundle by a sheaf supported on $E$.  Moving this sheaf to the beginning of the collection leaves line bundles in all remaining positions.  A common twist then makes their degrees on $E$ zero, so that they descend to $Z$.

The relevant decomposition $\Db(Y)=\langle\cO_E(-1),Lf^*\Db(Z)\rangle$ is the point-center case of \resultcite{Theorem~4.3}{Orl93}.  The calculations below identify the line bundles in its pullback component.  They apply the evaluation and coevaluation triangles of \resultcite{Definition~2.8 and Lemma~2.11}{Hu19}, with the shifts of Subsection~\ref{subsec:exceptional-mutations}, to the divisor and restriction exact sequences on $Y$.

\begin{lemma}\label{lem:3-29}
Let $\cL$ be a line bundle, $d=\cL\cdot E-1$, and $T=\cO_E(d)$.  Then $T$ is exceptional, and
\[
 \bL_{\cL}(\cL(E))\simeq T,
 \qquad\langle\cL,\cL(E)\rangle=\langle T,\cL\rangle.
\]
If a line bundle $\cM$ satisfies $\cM\cdot E=d$, then
\[
 \bR_T(\cM)\simeq\cM(-E),
 \qquad\langle\cM,T\rangle=\langle T,\cM(-E)\rangle.
\]
All displayed ordered pairs are exceptional.
\end{lemma}

\begin{proof}
The two-term Cartier-divisor resolution, as in \resultcite{proof of Tag~0AA4}{Stacks}, gives the sheaf Ext groups
\[
 \mathcal Ext_Y^q(T,T)=
 \begin{cases}
 \cO_E,&q=0,\\
 \cO_E(-1),&q=1,\\
 0,&q\notin\{0,1\}.
 \end{cases}
\]
Since $R\Gamma(\cO_E)\simeq\mathbb C$ and $R\Gamma(\cO_E(-1))=0$,
$\RHom_Y(T,T)\simeq\mathbb C$.  The sequences for $\cO_Y(\pm E)$ and
$R\Gamma(\cO_Y)\simeq\mathbb C$ give
\[
 \RHom(\cL(E),\cL)=R\Gamma(\cO_Y(-E))=0,
 \qquad
 \RHom(\cL,\cL(E))=R\Gamma(\cO_Y(E))\simeq\mathbb C.
\]
Thus the evaluation sequence
\begin{equation}\label{eq:3-51}
 0\longrightarrow\cL\longrightarrow\cL(E)\longrightarrow T\longrightarrow0
\end{equation}
identifies $\bL_{\cL}(\cL(E))$ with $T$.  Moreover,
\[
 \RHom(\cL,T)=R\Gamma(\mathbf P^1,\cO(-1))=0
\]
by Lemma~\ref{lem:3-26}.

For $\cM\cdot E=d$, the same lemma gives
\[
 \RHom(T,\cM)=0,
 \qquad\RHom(\cM,T)\simeq\mathbb C.
\]
The restriction $\cM|_E\simeq T$ yields the coevaluation sequence
\[
 0\longrightarrow\cM(-E)\longrightarrow\cM\longrightarrow T\longrightarrow0,
\]
so $\bR_T(\cM)\simeq\cM(-E)$.  Finally,
\[
 \cM(-E)\cdot E=d+1,
 \qquad\RHom(\cM(-E),T)=0.
\]
The exact sequences identify the thick spans; mutation preserves exceptionality.
\end{proof}

\begin{lemma}\label{lem:3-30}
Suppose $(\cO_E(d),\cM_1,\ldots,\cM_m)$ is a full exceptional collection on $Y$, with all $\cM_i$ line bundles.  There exist $\cN_i\in\Pic(Z)$ satisfying
\[
 f^*\cN_i\simeq\cM_i((d+1)E),
\]
and $(\cN_1,\ldots,\cN_m)$ is a full exceptional collection on $Z$.
\end{lemma}

\begin{proof}
Lemma~\ref{lem:3-26} gives
\[
 \RHom(\cM_i,\cO_E(d))=0\quad\Longrightarrow\quad\cM_i\cdot E=d+1.
\]
Tensoring the collection by $\cO_Y((d+1)E)$ gives
\[
 \cO_E(d)\otimes\cO_Y((d+1)E)\simeq\cO_E(-1),
 \qquad\cM_i((d+1)E)\cdot E=0.
\]
The point-blowup Picard decomposition \resultcite{\S1, before Lemma~1.2}{Roe01} implies
\[
 \Pic(Y)=f^*\Pic(Z)\oplus\mathbb Z[E],
 \qquad
 \mathcal P=f^*\cN\otimes\cO_Y(aE)
 \quad\Longrightarrow\quad\mathcal P\cdot E=-a.
\]
Hence $\cM_i((d+1)E)\simeq f^*\cN_i$, and
\begin{equation}\label{eq:3-52}
 \Db(Y)=\langle\cO_E(-1),f^*\cN_1,\ldots,f^*\cN_m\rangle.
\end{equation}
The following full-faithfulness calculation is the point-blowup specialization of \resultcite{Lemma~4.1 and its proof}{Orl93}.  Using $Rf_*\cO_Y\simeq\cO_Z$ and the projection formula,
\[
 Rf_*Lf^*F\simeq F,
 \qquad
 \RHom_Y(f^*\cN_j,f^*\cN_i)\simeq\RHom_Z(\cN_j,\cN_i).
\]
Thus $(\cN_i)$ is exceptional.  For every $F\in\Db(Z)$,
\begin{gather*}
 Lf^*F\in\langle\cO_E(-1),f^*\cN_1,\ldots,f^*\cN_m\rangle,\\
 Rf_*\cO_E(-1)=0,\qquad Rf_*f^*\cN_i\simeq\cN_i,\\
 F\simeq Rf_*Lf^*F\in\langle\cN_1,\ldots,\cN_m\rangle.
\end{gather*}
This proves fullness.
\end{proof}

\begin{proposition}\label{prop:3-31}
Let $f\colon Y\to Z$ contract a $(-1)$-curve $E$ on a smooth projective surface with $p_g(Y)=q(Y)=0$.  If a full exceptional collection of line bundles on $Y$ contains an adjacent pair $(\cL,\cL(E))$, then $Z$ admits a full exceptional collection of line bundles of length one less.
\end{proposition}

\begin{proof}
Write the collection as
\[
 (\cL_1,\ldots,\cL_{i-1},\cL,\cL(E),\cL_{i+2},\ldots,\cL_\gamma),
 \qquad d=\cL\cdot E-1,\quad T=\cO_E(d).
\]
Lemma~\ref{lem:3-29} replaces $(\cL,\cL(E))$ by $(T,\cL)$.  For every preceding line bundle, semiorthogonality and Lemma~\ref{lem:3-26} give
\[
 \RHom(T,\cL_j)=0\quad\Longrightarrow\quad\cL_j\cdot E=d
 \qquad(j<i).
\]
Apply the replacement $(\cL_j,T)\mapsto(T,\cL_j(-E))$ of Lemma~\ref{lem:3-29}, in the order $j=i-1,\ldots,1$.  The resulting full exceptional collection is
\[
 (T,\cL_1(-E),\ldots,\cL_{i-1}(-E),\cL,\cL_{i+2},\ldots,\cL_\gamma).
\]
It consists of $T$ followed by $\gamma-1$ line bundles.  Apply Lemma~\ref{lem:3-30}.
\end{proof}

\subsection{Proof of the rationality theorem}\label{subsec:main-proof}

The contraction step preserves the hypotheses needed to find the next adjacent pair: the base surface $S$ stays fixed, and the remaining centers remain distinct and avoid every rational curve on $S$.  Thus the collection can be descended until no blowup centers remain.

\begin{proof}[Proof of Theorem~\ref{thm:main}]
Suppose that $X$ is nonrational and admits a full exceptional collection of line bundles.  Let $X\to S$ be a tower of $r$ point blowups with $S$ minimal.  Exceptionality and birational invariance give
\[
 p_g(S)=p_g(X)=0,\qquad q(S)=q(X)=0.
\]
By the surface minimal model theorem, a minimal surface is either of nef canonical class or is $\mathbf P^2$ or ruled; in characteristic zero, uniruled surfaces are birationally ruled \resultcite{\S\S5.6 and~9.3}{Deb11}.  A ruled surface with $q=0$ is rational.  Thus
\[
 X\text{ nonrational}\quad\Longrightarrow\quad
 S\text{ non-uniruled},\qquad K_S\text{ nef}.
\]
Corollary~\ref{cor:3-23} gives a full exceptional collection of line bundles on
$Y=\operatorname{Bl}_{p_1,\ldots,p_r}S$, where the centers are distinct and avoid all integral rational curves of $S$.

For $I\subseteq\{1,\ldots,r\}$, write
\[
 Y_I=\operatorname{Bl}_{\{p_a:a\in I\}}S,
 \qquad Y_\varnothing=S,
\]
and let $P(I)$ denote the existence of a full exceptional collection of line bundles on $Y_I$.  We have $P(\{1,\ldots,r\})$.  For every nonempty $I$, Propositions~\ref{prop:3-28} and~\ref{prop:3-31} give
\[
 P(I)\quad\Longrightarrow\quad P(I\setminus\{a\})
 \quad\text{for some }a\in I.
\]
The remaining centers retain the avoidance property.  Induction on $|I|$ therefore yields $P(\varnothing)$, including the case $r=0$.

Let $\cE_S$ be the resulting full collection on $S$.  Since $\Db(S)\ne0$, this collection is nonempty.  Propositions~\ref{prop:2-2} and~\ref{prop:3-16} imply
\[
 0=h_S(\cE_S)\ge1,
\]
a contradiction.  Hence $X$ is rational.
\end{proof}

\begin{proof}[Proof of Corollary~\ref{cor:characterization}]
Theorem~\ref{thm:main} gives the implication from a full exceptional collection of line bundles to rationality.  The converse is \resultcite{Theorem~5.6}{HP11}.
\end{proof}

\clearpage
\appendix
\section{The rationality theorem in arbitrary characteristic}\label{sec:arbitrary-characteristic}

The argument of the main text extends to algebraically closed fields of arbitrary characteristic. We use its characteristic-independent proofs and give the replacements for nilpotent units, the choice of blowup centers, and the minimal-model input.

\begin{theorem}\label{thm:arbitrary-characteristic}
Let $k$ be an algebraically closed field of arbitrary characteristic, and let $X$ be a smooth connected projective surface over $k$. If $\Db(X)$ admits a full exceptional collection of line bundles, then $X$ is rational.
\end{theorem}

Together with \resultcite{Theorem~5.6}{HP11}, whose construction applies over an algebraically closed field of any characteristic, this gives the same characterization as Corollary~\ref{cor:characterization}. Over a ground field that is not algebraically closed, the conclusions are as follows.

\begin{corollary}\label{cor:arbitrary-ground-field}
Let $k$ be a field, and let $X$ be a smooth geometrically integral projective surface over $k$. If $\Db(X)$ admits a full exceptional collection of line bundles, then $X$ is geometrically rational. If $k$ is perfect, then $X$ is $k$-rational.
\end{corollary}

In this corollary, exceptionality means $\RHom(E,E)\simeq k$. The assertion over perfect fields combines Theorem~\ref{thm:arbitrary-characteristic} with \resultcite{Theorem~3.7}{Via17}. The proof is given in Subsection~\ref{subsec:appendix-ground-fields}.

\subsection{Prior work over general fields}\label{subsec:appendix-literature}

The general-field analogue of Conjecture~\ref{conj:orlov} asks whether a smooth geometrically integral projective surface with a full $k$-exceptional collection is $k$-rational. Even when a surface is known to be geometrically rational, constructing a birational parametrization over $k$ is a further problem. The arithmetic literature also considers \emph{\'etale-exceptional} objects, for which $\RHom(E,E)\simeq L$ with $L/k$ a finite separable field extension \resultcite{Definition~2.2}{BDLM24}. A full collection of these objects is equivalent to categorical representability in dimension zero \resultcite{Lemma~10}{BD21}.

Brauer--Severi varieties provide a first test of the conjecture. Their derived categories have semiorthogonal decompositions into categories of modules over tensor powers of a central simple algebra \resultcite{Corollary~4.7}{Ber09}. Noncommutative motives show that a nonsplit Brauer--Severi variety admits no full \'etale-exceptional collection \resultcite{Theorem~3.1}{Rae16}. Thus a full $k$-exceptional collection forces a Brauer--Severi variety to be projective space. In particular, the conjecture holds for forms of $\mathbb P^2$ over every field.

For geometrically rational surfaces over perfect fields, a numerical rationality criterion is available. A numerically exceptional collection of maximal length implies $k$-rationality if the surface is minimal, or if the collection consists of numerical line bundles \resultcite{Theorem~3.7}{Via17}. The argument uses the unimodularity of the intersection lattice: every blowup in a morphism to a smooth minimal model must then have a $k$-rational center \resultcite{Proposition~2.3 and Lemma~3.8}{Via17}. This makes the numerical criterion particularly suited to the passage from geometric rationality to rationality over the ground field. The converse already fails for the blowup of $\mathbb P^2_k$ at a separable closed point of degree greater than one, which is $k$-rational but admits no numerically exceptional collection of maximal length \resultcite{Remark~3.9}{Via17}.

A complementary approach studies the components of the derived category of a del Pezzo surface. Over a perfect field, a del Pezzo surface of Picard rank one is $k$-rational precisely when it is categorically representable in dimension zero \resultcite{Theorem~1}{AB18}. For degree at least five, the same criterion holds without a Picard-rank restriction and is also equivalent to the existence of a $k$-point. In this range, explicit decompositions describe the relevant components by semisimple algebras and their Brauer classes \resultcite{Theorems~1 and~2}{AB18}. For conic bundle surfaces, the natural decomposition involving the even Clifford algebra is used to construct a Griffiths--Kuznetsov component \resultcite{\S3.2, Definition~12, and Theorem~1}{BD21}. Together with the del Pezzo case, this gives a categorical birational invariant for geometrically rational surfaces over perfect fields \resultcite{Theorem~2}{BD21}.

Galois descent supplies collections on many forms of toric varieties. A full Galois-stable exceptional collection over a finite Galois extension descends to a full collection over $k$ when finite-dimensional division algebras are allowed as endomorphism algebras \resultcite{Definition~2.2 and Theorem~1.3}{BDM19}. In particular, every smooth projective toric surface over $k$ admits a full exceptional collection of sheaves in this sense \resultcite{Proposition~4.7}{BDM19}. The descended endomorphism algebras retain arithmetic information that disappears over a splitting field.

For full $k$-exceptional collections, rationality has been proved for smooth projective arithmetic toric varieties of any dimension having a $k$-point \resultcite{Theorem~1}{BDLM24}. The proof passes from the collection to the trivial Galois action on geometric $K_0$ and the Picard group, and then to rationality of the dense torus \resultcite{\S5.1}{BDLM24}. The \'etale-exceptional version has a different outcome in dimension three: geometrically rational smooth projective threefolds over $\mathbb Q$ have been constructed with full \'etale-exceptional collections that are nonrational despite having a $\mathbb Q$-point, and also with no $\mathbb Q$-point at all \resultcite{Theorems~2 and~3}{BDLM24}.

For the line-bundle case considered here, Theorem~\ref{thm:arbitrary-characteristic} supplies the geometric rationality needed in the numerical criterion above, in every characteristic. Base change therefore gives geometric rationality over an arbitrary field, and \resultcite{Theorem~3.7}{Via17} then gives $k$-rationality over a perfect field. Subsection~\ref{subsec:appendix-ground-fields} proves Corollary~\ref{cor:arbitrary-ground-field} by this argument.

In Subsections~\ref{subsec:appendix-divisors}--\ref{subsec:appendix-proof}, $k$ is algebraically closed. All cohomology and derived Hom groups are taken over $k$, and the notation and shift convention of Section~\ref{sec:preliminaries} are retained. In particular, an exceptional line bundle gives
\begin{equation}\label{eq:appendix-structure}
H^0(X,\cO_X)=k,\qquad H^1(X,\cO_X)=H^2(X,\cO_X)=0.
\end{equation}

\subsection{Line bundles on effective left-orthogonal divisors}\label{subsec:appendix-divisors}

The proofs of Lemmas~\ref{lem:3-1}--\ref{lem:3-2} use restriction sequences, Riemann--Roch, normalization, and intersection theory, and apply over $k$ without change. The same is true of the reduced-support and node-gluing arguments in Lemma~\ref{lem:3-3}: the reduced support of an effective left-orthogonal divisor is a nodal tree of projective lines, and its Picard group is determined by component degrees. The restriction and normalization arguments of \resultcite{Lemma~4.1, Proposition~4.3, and the proof of Theorem~4.5}{Ela16}, stated there in characteristic zero, apply over $k$.

For the nonreduced step, we also use the quotient argument of Lemma~\ref{lem:3-1} for an arbitrary closed subscheme $C\subseteq D$. The kernel of $\cO_D\to\cO_C$ is supported in dimension at most one, so $H^1(C,\cO_C)=0$. The finite logarithm in Lemma~\ref{lem:3-3} is the only step in that proof that needs replacement. We give a Picard-scheme argument for it.

\begin{lemma}\label{lem:appendix-picard}
Let $D>0$ be an effective left-orthogonal divisor, with reduced components $C_1,\ldots,C_v$. Restriction and multidegree induce isomorphisms
\begin{equation}\label{eq:appendix-picard}
\Pic(D)\xrightarrow{\sim}\Pic(D_{\mathrm{red}})
\xrightarrow{\sim}\mathbb Z^v.
\end{equation}
If $\cN\in\Pic(D)$ has nonnegative degree on every $C_i$, it has a section that is a nonzerodivisor on $D$. For every $\cM\in\Pic(D)$, it follows that
\begin{equation}\label{eq:appendix-twist-vanishing}
H^1(D,\cM)=0\quad\Longrightarrow\quad H^1(D,\cM\otimes\cN)=0.
\end{equation}
\end{lemma}

\begin{proof}
Every intermediate closed subscheme $D_{\mathrm{red}}\subseteq C\subseteq D$ has $H^1(C,\cO_C)=0$, as observed above. Its Picard scheme exists and is locally of finite type by \resultcite{Corollary~4.18.3}{Kle05}, and its tangent space at the identity is $H^1(C,\cO_C)$ by \resultcite{Theorem~5.11}{Kle05}. Hence this Picard scheme is discrete and reduced. Indeed, its local ring at the identity has zero cotangent space, so its maximal ideal is zero by Nakayama's lemma. Translation gives the same conclusion at every closed point. Since the scheme is locally of finite type over the algebraically closed field $k$, these points exhaust its zero-dimensional components.

Filter the nilradical of $\cO_D$ by its powers. Each successive step is a square-zero immersion $C\subset C'$ with ideal $\cJ$. The exact sequence
\[
1\longrightarrow1+\cJ\longrightarrow\cO_{C'}^*
\longrightarrow\cO_C^*\longrightarrow1
\]
identifies $1+\cJ$ with the additive sheaf $\cJ$; this uses only $\cJ^2=0$. Its cohomology sequence is \resultcite{Tag~0C6R}{Stacks}. Sheafifying in the Picard functor gives an exact sequence of fppf sheaves
\begin{equation}\label{eq:appendix-square-zero}
\mathbf V\bigl(H^1(C,\cJ)\bigr)\longrightarrow
\Pic_{C'/k}\longrightarrow\Pic_{C/k}\longrightarrow0.
\end{equation}
Here $\mathbf V(W)$ denotes the additive vector group associated with a finite-dimensional $k$-vector space $W$. Over an affine test scheme $\operatorname{Spec}A$, the cohomology of $\cJ_A$ is $H^i(C,\cJ)\otimes_k A$. The obstruction group $H^2(C,\cJ)$ vanishes because $C$ is one-dimensional. The long exact sequence therefore gives \eqref{eq:appendix-square-zero}.

A homomorphism from a connected vector group to the discrete reduced group scheme $\Pic_{C'/k}$ is zero. Thus \eqref{eq:appendix-square-zero} makes restriction an isomorphism. Iteration proves the first map in \eqref{eq:appendix-picard}. Over an algebraically closed field, the points of these Picard schemes are the ordinary Picard groups. The characteristic-independent node-gluing argument of Lemma~\ref{lem:3-3} proves the second map in \eqref{eq:appendix-picard}.

With the multidegree isomorphism established, the final Cartier-divisor construction and cohomology sequence of Lemma~\ref{lem:3-3} apply over $k$ without change. They give the nonzerodivisor section of $\cN$ and the implication \eqref{eq:appendix-twist-vanishing}.
\end{proof}

\begin{proposition}\label{prop:appendix-multiplication}
Lemmas~\ref{lem:3-1}--\ref{lem:3-2}, Proposition~\ref{prop:3-4}, and Proposition~\ref{prop:3-8} hold over an algebraically closed field of arbitrary characteristic. In particular, under the canonical-degree condition \eqref{eq:3-3}, $R=A^0$ is a product of path algebras of linearly oriented chains, and $A^1$ is projective as both a left and a right $R$-module.
\end{proposition}

\begin{proof}
Lemmas~\ref{lem:3-1}--\ref{lem:3-2} were addressed above. In the proof of Proposition~\ref{prop:3-4}, replace Lemma~\ref{lem:3-3} by Lemma~\ref{lem:appendix-picard}. Every other step is unchanged: restriction sequences, Cartier-divisor adjunction, and Serre duality on the Gorenstein curve $D$ are valid over $k$; see \resultcite{Tags~0AA4 and~0BS2(5)}{Stacks}.

The proof of Proposition~\ref{prop:3-8} now applies with these inputs. Its graph argument and the normalization of Hom generators at a point outside finitely many divisors work over any algebraically closed field. The chain-representation arguments of Lemmas~\ref{lem:3-5}--\ref{lem:3-7} are linear algebra over any field; their sources are \resultcite{\S1, pp.~4 and~7--8}{CB92} and \resultcite{Example~4.2(i) and Theorem~5.4}{LZ11}. Thus the same proof gives the asserted description of $R$ and projectivity of $A^1$ on both sides.
\end{proof}

\subsection{The height obstruction}\label{subsec:appendix-height}

\begin{proposition}\label{prop:appendix-height}
Let $Y$ be a smooth connected projective surface over $k$, and let $\cE=(\cL_1,\ldots,\cL_\gamma)$ be a nonempty exceptional collection of line bundles. Assume that $K_Y$ has nonnegative degree on every prime component of the zero divisor of every nonzero forward morphism. Then
\[
\NHH^u(\cE,Y)=0\qquad(u<1).
\]
In particular, $\cE$ is not full. This applies whenever $K_Y$ is nef.
\end{proposition}

\begin{proof}
Normal Hochschild cohomology, its localization triangle, and the finite directed bar spectral sequence are defined over $k$; see \resultcite{Theorem~3.3, Lemma~3.6, and Proposition~3.7}{Kuz15}. The proofs of Lemmas~\ref{lem:3-9}--\ref{lem:3-15} are independent of characteristic: the idempotent algebra is the split algebra $S_0=\bigoplus_i ke_i$, and the bar resolutions and their contracting homotopies use no division. The changes of signs in Lemma~\ref{lem:3-10} remain valid in characteristic two. Use Proposition~\ref{prop:appendix-multiplication} in place of Proposition~\ref{prop:3-8}. The proof of Proposition~\ref{prop:3-16} then gives the stated vanishing, since the inverse-Serre coefficient bimodule still satisfies $T^t=0$ for $t<2$ by \eqref{eq:2-1}.

For the implication to nonfullness, replace the Hochschild--Kostant--Rosenberg argument in Proposition~\ref{prop:2-2} by the diagonal-kernel calculation
\begin{equation}\label{eq:appendix-hh-zero}
\HH^0(Y)=\Hom_{Y\times Y}(\cO_\Delta,\cO_\Delta)
=H^0(Y,\cO_Y)=k.
\end{equation}
If the collection were full, the localization triangle \resultcite{Theorem~3.3}{Kuz15} would identify its normal Hochschild cohomology with $\HH(Y)$, contradicting the vanishing in degree zero.
\end{proof}

\subsection{Geometric generic configurations of centers}\label{subsec:appendix-generic-centers}

The avoidance of all rational curves in Subsection~\ref{subsec:general-centers} uses characteristic zero and the uncountability of $\mathbb C$. Over $k$, we impose the incidence condition of Lemma~\ref{lem:appendix-generic-avoidance} on every partial blowup by passing to the geometric generic point of the configuration space.

Let $S$ be a smooth connected projective surface over $k$ with $K_S$ nef. For $r\geq1$, let $\Omega$ be an algebraic closure of $k(\Conf_r(S))$, and let $(p_1,\ldots,p_r)$ be the corresponding geometric generic configuration on $S_\Omega$. For $I\subseteq\{1,\ldots,r\}$ set
\begin{equation}\label{eq:appendix-partial-blowups}
Y_I=\operatorname{Bl}_{\{p_i\,\mid\,i\in I\}}S_\Omega,
\qquad Y_\varnothing=S_\Omega,
\end{equation}
with exceptional curves $E_i$ for $i\in I$. Nefness persists after extending the algebraically closed ground field. One can check this by spreading a curve to a finite-type parameter scheme and specializing to a closed $k$-point: flatness preserves its degree against $K_S$, so a negative degree after extension would give a negative degree over $k$.

\begin{lemma}\label{lem:appendix-generic-avoidance}
For every $I$ and every $a\in I$, no nonexceptional smooth rational curve $C\subset Y_I$ satisfies $CE_a=1$.
\end{lemma}

\begin{proof}
First consider the last point of a geometric generic configuration of size $m$. Put $B=\Conf_{m-1}(S)$, with $B=\operatorname{Spec}k$ if $m=1$, and let $\mathcal Z\to B$ be the family of blowups at its $m-1$ tautological disjoint sections. The blowup-family construction is the distinct-center case of \resultcite{Propositions~1.4--1.5}{Roe01}.

Every smooth rational curve $C$ on a geometric fiber $Z$ has negative self-intersection. If it is exceptional, its square is $-1$. Otherwise, with $\pi\colon Z\to S$ denoting the geometric base change of the blowup morphism,
\[
K_ZC=K_S\pi_*C+\sum_{i=1}^{m-1}E_iC\geq0,
\qquad C^2=-2-K_ZC\leq-2.
\]
Here and below $K_S$ denotes its pullback to the relevant field. Hence $H^0(C,N_{C/Z})=0$.

Choose a relatively ample line bundle. The relative Hilbert scheme has finite-type pieces for fixed Hilbert polynomial, and relative effective Cartier divisors form an open subfunctor; see \resultcite{Theorem~3.7 and Answer~3.8, p.~64}{Kle05}. In each piece, let $H$ parametrize smooth geometrically connected curves of genus zero. This is an open locus: properness of the universal family makes smoothness of the entire fiber an open condition, and on that locus connectedness and genus zero are equivalent to vanishing of the cone of $\cO\to Rg_*\cO$ for the universal curve $g$. This cone is perfect and commutes with base change \resultcite{Tag~0DJT}{Stacks}, so its vanishing is open by derived Nakayama \resultcite{Tag~0BCD}{Stacks}. The tangent space of a fiber of $H\to B$ at $[C]$ is
\[
H^0(C,N_{C/Z});
\]
this is the Hilbert-scheme tangent-space calculation in \resultcite{Remark~5.18, equation~(5.18.1)}{Kle05}. It vanishes by the preceding paragraph. Each fiber of $H\to B$ is consequently zero-dimensional. Since the Hilbert scheme for a fixed polynomial is of finite type, the map is quasi-finite and
\[
\dim H\leq\dim B=2m-2.
\]
Its universal curve therefore has dimension at most $2m-1$.

On this universal curve, restrict the marked point to the complement of the exceptional divisors and blow it down to $S$. Recording the preceding $m-1$ centers and this additional point gives an incidence morphism to $\Conf_m(S)$. The closure of its image has dimension at most $2m-1$, less than $\dim\Conf_m(S)=2m$. For each Hilbert polynomial this is a proper closed subset. A geometric generic point lies over the generic point of $\Conf_m(S)$ and belongs to none of these subsets. Thus, on the blowup at the preceding centers, no smooth rational curve passes through the new center.

Now suppose a nonexceptional smooth rational curve on the final blowup meets the last exceptional curve once. This is a single transverse intersection. After contracting the exceptional curve, its image has multiplicity one at the center and is smooth there; it is unchanged elsewhere. The image is therefore a smooth rational curve through the new center, contradicting the incidence exclusion.

Finally, every projection $\Conf_r(S)\to\Conf_I(S)$ is dominant. The image of our configuration is a geometric generic configuration for this smaller space, after a further extension of the function field. The exclusions just proved persist under that extension. Permuting indices makes any prescribed $a\in I$ the last point. This proves the assertion simultaneously for all $I$ and $a$.
\end{proof}

\begin{proposition}\label{prop:appendix-generic-divisors}
Suppose also that $H^1(S,\cO_S)=H^2(S,\cO_S)=0$. Every effective left-orthogonal divisor $D>0$ on $Y_I$ satisfies exactly one of the following:
\begin{enumerate}[label=\textup{(\roman*)}]
\item $D=E_a$ for some $a\in I$;
\item $D$ has no exceptional component, $K_{Y_I}$ is nonnegative on every component of $D$, and $\chi(\cO_{Y_I}(D))\leq0$.
\end{enumerate}
\end{proposition}

\begin{proof}
The structure-sheaf cohomology is unchanged by point blowups and field extension. For the blowup assertion one uses $R\pi_*\cO\simeq\cO$, the calculation in \resultcite{Lemma~4.1 and its proof}{Orl93}; compare also \resultcite{Proposition~4.1}{Lie13}. Thus the divisor arguments of Subsection~\ref{subsec:appendix-divisors} apply on $Y_I$.

If $D$ contains an exceptional curve $E_a$ and another component, the reduced-support argument of Lemma~\ref{lem:3-3}, valid here by Subsection~\ref{subsec:appendix-divisors}, gives a component $C$ neighboring $E_a$ in the nodal tree, with $CE_a=1$. Since the exceptional curves are disjoint, $C$ is nonexceptional. It is a smooth rational curve, contrary to Lemma~\ref{lem:appendix-generic-avoidance}. If $E_a$ is the only component, write $D=nE_a$. The identity $D^2+K_{Y_I}D=-2$ gives $-n^2-n=-2$, and hence $n=1$.

Otherwise there is no exceptional component. For every component $C$,
\[
K_{Y_I}C=K_{S_\Omega}\pi_*C+\sum_{a\in I}E_aC\geq0.
\]
Finally, \eqref{eq:2-3} gives $\chi(\cO_{Y_I}(D))=-K_{Y_I}D\leq0$.
\end{proof}

This proposition supplies the canonical-degree and Euler-characteristic inequalities used in the extraction and adjacency arguments of Subsection~\ref{subsec:exceptional-pair}.

\subsection{Deforming the collection to the generic configuration}\label{subsec:appendix-deformation}

\begin{lemma}\label{lem:appendix-fullness}
Let $f\colon\mathcal X\to B$ be smooth and projective over a locally noetherian base, with geometrically connected fibers, and let $\widetilde{\cL}_1,\ldots,\widetilde{\cL}_\gamma$ be line bundles on $\mathcal X$. If their restrictions are full and exceptional at a geometric point of $B$, the same is true at every geometric point of an open neighborhood. A full exceptional collection of line bundles also remains full and exceptional after any extension of its ground field.
\end{lemma}

\begin{proof}
Openness is \resultcite{Proposition~3.9}{Hu19}. The kernel proof of Proposition~\ref{prop:3-19} applies in this setting, using proper perfect pushforward and derived base change \resultcite{Tag~0DJT}{Stacks}, convolution \resultcite{Tag~0FYP}{Stacks}, and derived Nakayama \resultcite{Tag~0BCD}{Stacks}. The skyscraper-sheaf proof of Lemma~\ref{lem:3-17} applies over each algebraically closed residue field, and the projection construction of Lemma~\ref{lem:3-18} uses only semiorthogonality. Thus the same unit cone, reverse Hom complexes, and projection kernel give the required open neighborhood.

For the ground-field assertion, start with the same projection kernel on a smooth projective variety over a field $F$. If its functor is zero, apply it to the residue field of each closed point $x$. The resulting object is the pushforward of the derived restriction of the kernel to $x\times_F X$. The projection $x\times_F X\to X$ is finite, and its pushforward is exact and faithful. These restrictions therefore vanish, and derived Nakayama gives vanishing of the kernel. The zero kernel stays zero after any extension $F'/F$. Its convolution description commutes with that flat base change, as do the exceptional Hom complexes. The collection over $F'$ is consequently full and exceptional.
\end{proof}

\begin{proposition}\label{prop:appendix-generic-fullness}
Let $S$ be a smooth connected projective surface over $k$, and let $X\to S$ be an ordered tower of $r\geq1$ point blowups, possibly with infinitely near centers. A full exceptional collection of line bundles on $X$ induces one on the blowup of $S_\Omega$ at a geometric generic configuration of $r$ distinct points. Here $\Omega$ is an algebraic closure of the function field of $\Conf_r(S)$.
\end{proposition}

\begin{proof}
Apply Lemma~\ref{lem:3-20} and its proof, using relative point blowups and the Picard decomposition \eqref{eq:3-46}. These constructions are valid over $k$; see \resultcite{Propositions~1.4--1.5, \S1 before Lemma~1.2, and \S2 before Remark~2.1}{Roe01}. They give the smooth irreducible ordered-cluster base $B_r$, with dense open subset $\Conf_r(S)$, and extensions of all the given line bundles from the original closed fiber.

By Lemma~\ref{lem:appendix-fullness}, the extended collection is full and exceptional on a nonempty open subset of $B_r$. This subset contains the generic point of $\Conf_r(S)$. Passing to its geometric generic point proves the assertion.
\end{proof}

\subsection{Extraction and contraction of an exceptional-curve pair}\label{subsec:appendix-contraction}

Fix a partial geometric generic blowup $Y_I$ as in \eqref{eq:appendix-partial-blowups}, with $H^1(S,\cO_S)=H^2(S,\cO_S)=0$ and $K_S$ nef. We apply the divisor alternative of Proposition~\ref{prop:appendix-generic-divisors} to extract and contract an adjacent exceptional-curve pair.

\begin{proposition}\label{prop:appendix-contraction}
If $I\ne\varnothing$ and $Y_I$ has a full exceptional collection of line bundles, the collection contains an adjacent pair $(\cL,\cL(E_a))$ for some $a\in I$. Contracting $E_a$ gives a full exceptional collection of line bundles on $Y_{I\setminus\{a\}}$, of length one less.
\end{proposition}

\begin{proof}
Proposition~\ref{prop:appendix-generic-divisors} supplies the two consequences of Lemma~\ref{lem:3-24} used in the extraction argument: if an effective left-orthogonal divisor is not an $E_a$, its component canonical degrees are nonnegative and its Euler characteristic is nonpositive. Consequently the proofs of Lemma~\ref{lem:3-27} and Corollary~\ref{cor:3-25} apply, using Proposition~\ref{prop:appendix-height} in place of Proposition~\ref{prop:3-16}. They give an exceptional-curve difference and the same criterion for a nonzero forward Hom complex to be concentrated in degree zero.

Now apply the minimal-index-gap proof of Proposition~\ref{prop:3-28}. Its remaining inputs are the Cartier-divisor adjunction formulas of Lemma~\ref{lem:3-26} and line-bundle cohomology on $\mathbf P^1$, both valid in every characteristic. This yields an adjacent pair $(\cL,\cL(E_a))$.

The proof of Proposition~\ref{prop:3-31} gives descent and the length assertion without change. Indeed, Lemmas~\ref{lem:3-29}--\ref{lem:3-30} use explicit divisor and restriction sequences, the point-blowup Picard decomposition \resultcite{\S1, before Lemma~1.2}{Roe01}, and the pushforward identities in \resultcite{Lemma~4.1 and its proof}{Orl93}. These constructions are valid over $k$.
\end{proof}

\subsection{The minimal model and the proof of rationality}\label{subsec:appendix-proof}

\begin{lemma}\label{lem:appendix-minimal}
Let $X$ be a nonrational smooth connected projective surface over $k$ satisfying \eqref{eq:appendix-structure}. It has a smooth minimal model $S$ with the same cohomology vanishings and with $K_S$ nef.
\end{lemma}

\begin{proof}
Successive contraction of $(-1)$-curves produces a smooth projective minimal model; each contraction decreases the Picard number. In positive characteristic this is Castelnuovo's contraction theorem \resultcite{Theorem~4.2}{Lie13}; for characteristic zero see \resultcite{\S5.6}{Deb11}. The structure-sheaf cohomology is unchanged by the point blowups, as in the proof of Proposition~\ref{prop:appendix-generic-divisors}.

If $\kappa(S)=-\infty$, the classification of surfaces makes $S$ birationally ruled. In positive characteristic, the precise statement, including $h^1(S,\cO_S)=g(C)$ for a ruling curve $C$, is \resultcite{Theorem~4.4}{Lie13}; the characteristic-zero statement follows from \resultcite{\S5.6}{Deb11}. Thus $h^1=0$ would make $S$ rational, a contradiction. Hence $\kappa(S)\geq0$.

Adjunction and minimality give nefness. Choose an effective divisor linearly equivalent to $mK_S$ for some $m>0$. If $K_SC<0$ for an integral curve $C$, it is a component of that divisor and has $C^2<0$. Adjunction gives
\[
2p_a(C)-2=C^2+K_SC.
\]
Both integers on the right are negative, while the left is at least $-2$. They must both equal $-1$, and $p_a(C)=0$. Normalization then shows that $C$ is a smooth rational $(-1)$-curve, contradicting minimality. Therefore $K_S$ is nef.
\end{proof}

\begin{proof}[Proof of Theorem~\ref{thm:arbitrary-characteristic}]
We apply the induction of Subsection~\ref{subsec:main-proof} with the replacements established above. Suppose that $X$ is nonrational and has a full exceptional collection of line bundles. Exceptionality gives \eqref{eq:appendix-structure}, and Lemma~\ref{lem:appendix-minimal} gives a tower $X\to S$ of $r$ point blowups with $K_S$ nef.

If $r>0$, Proposition~\ref{prop:appendix-generic-fullness} supplies a full exceptional collection on the corresponding geometric generic blowup of $S_\Omega$. Lemma~\ref{lem:appendix-generic-avoidance} applies to every subset of its centers. Thus Proposition~\ref{prop:appendix-contraction} supplies each step of the same induction on the number of remaining centers, and yields a full exceptional collection of line bundles on $S_\Omega$. For $r=0$, take $\Omega=k$ and use the original collection.

The resulting collection is nonempty because $\Db(S_\Omega)\ne0$. Nefness persists under field extension, so Proposition~\ref{prop:appendix-height} contradicts fullness on $S_\Omega$. Hence $X$ is rational.
\end{proof}

\subsection{The ground field}\label{subsec:appendix-ground-fields}

\begin{proof}[Proof of Corollary~\ref{cor:arbitrary-ground-field}]
Let $\bar k$ be an algebraic closure of $k$. By Lemma~\ref{lem:appendix-fullness}, the collection remains full and exceptional on $X_{\bar k}$. Theorem~\ref{thm:arbitrary-characteristic} makes $X_{\bar k}$ rational. This proves geometric rationality over every field.

Suppose now that $k$ is perfect. The collection on $X$ is numerically exceptional and consists of numerical line bundles. It also has maximal numerical length. Indeed, the successive exceptional projections of Subsection~\ref{subsec:exceptional-mutations} give
\[
K_0(X)=\bigoplus_{i=1}^\gamma\mathbb Z[\cL_i].
\]
The Euler matrix is upper triangular with diagonal entries one, so its determinant is one and its radical is zero. Hence $\gamma=\operatorname{rank}K_0^{\mathrm{num}}(X)$. Surface Riemann--Roch identifies this rank with $\operatorname{rank}N^1(X)+2$: after tensoring with $\mathbb Q$, the numerical Chern character has rank, divisor class, and degree components, with nondegenerate intersection pairing on numerical divisor classes.

Thus $X$ is geometrically rational over a perfect field and admits a numerically exceptional collection of maximal length consisting of numerical line bundles. These are the hypotheses of the second alternative in \resultcite{Theorem~3.7}{Via17}, which gives $k$-rationality.
\end{proof}

\end{document}